\documentclass[11pt]{article}

\usepackage[pagebackref,colorlinks=true,pdfpagemode=none,urlcolor=blue,
linkcolor=blue,citecolor=blue]{hyperref}

\usepackage{amsmath,amsfonts,amssymb,amsthm}
\usepackage{mathrsfs}
\usepackage{color}

\newtheorem{theorem}{Theorem}[section]
\newtheorem{lemma}[theorem]{Lemma}

\newtheorem{proposition}[theorem]{Proposition}

\theoremstyle{remark}
\newtheorem{remark}[theorem]{Remark}

\renewenvironment{proof}[1][Proof]{ {\itshape \noindent {#1.}} }{$\Box$
\medskip}

\numberwithin{equation}{section}
\newcommand{\R}{\mathbb{R}}
\newcommand{\CC}{\mathbb{C}}

\newcommand{\Pb}{\mathbb{P}}
\newcommand{\E}{\mathbb{E}}
\newcommand{\F}{\mathcal{F}}
\newcommand{\B}{\mathcal{B}}
\newcommand{\Li}{\mathcal{L}}
\newcommand{\G}{\mathcal{G}}

\newcommand{\A}{\mathcal{A}}
\newcommand{\Tr}{\mathrm{Tr}}
\newcommand{\C}{\mathcal{C}}
\newcommand{\K}{\mathcal{K}}

\newcommand{\V}{\mathcal{V}}
\newcommand{\Q}{\mathcal{Q}}
\newcommand{\eps}{\varepsilon}
\def\les{\lesssim}

\begin{document}
\title{Radiative Transport Limit of Dirac Equations with Random Electromagnetic Field}
\author{Yu Gu\thanks{Department of Applied Physics \& Applied
Mathematics, Columbia University, New York, NY 10027 (yg2254@columbia.edu; gb2030@columbia.edu)}  \and Guillaume Bal\footnotemark[1]}
\maketitle

\begin{abstract}
This paper concerns the kinetic limit of the Dirac equation with random electromagnetic field. We give a detailed mathematical analysis of the radiative transport limit for the phase space energy density of solutions to the Dirac equation. Our derivation is based on a martingale method and a perturbed test function expansion. This requires the electromagnetic field to be a space-time random field.

The main mathematical tool in the derivation of the kinetic limit is the matrix-valued Wigner transform of the vector-valued Dirac solution. The major novelty compared to the scalar (Schr\"odinger) case is the proof of convergence of cross-modes to 0 weakly in space and almost surely in probability. The propagating modes are shown to converge in an appropriate strong sense to their deterministic limit.
\end{abstract}

\section{Introduction}
The Dirac equation is the relativistic version of the Schr{\"o}dinger equation and describes very fast electrons propagating in an electromagnetic field. In this paper, we consider the semiclassical limit of the Dirac equation when the electromagnetic field is random and time-dependent.

The problem falls into the category of high frequency wave propagating in highly heterogeneous media, which has been modeled by radiative transfer equation in many areas, e.g., quantum waves in semiconductors, electromagnetic waves in turbulent atmospheres and plasmas, underwater acoustic waves, elastic waves in the Earth's crust. Such kinetic models account for the multiple interactions of wave fields with the fluctuation of the underlying media. In the so-called weak-coupling limit we will consider,  waves propagate over distances that are large compared to the typical wavelength in the system and the fluctuations have weak amplitude with correlation length comparable to the wavelength. Most of the derivations of radiative transfer equation are based on formal expansions, e.g., a systematic method to derive kinetic equations from symmetric first-order hyperbolic systems, including systems of acoustic and elastic equations, in the weak-coupling limit has been presented in \cite{ryzhik1996transport} and extended in various forms in \cite{bal2005kinetics,bal1999radiative,guo1999transport,powell2005transport}.

Mathematically rigorous derivations are notoriously difficult in the setting of spatially varying randomness. Most proofs have been obtained for the Schr\"odinger equation with time-independent Gaussian potential, see \cite{erdHos2000linear,spohn1977derivation}, with an extension to (discrete) wave equations in \cite{lukkarinen2007kinetic}, and are based on the Neumann series expansion for the solution to the Schr{\"o}dinger equation and appropriate estimates that allow passage to the limit. Such techniques were also used in \cite{BKR-ARMA-11} to obtain the random mixing of the phase of the Schr\"odinger solution and in \cite{B-CMP-09,GB2012schrodingerRandom,GB2012schrodingerDeterministic,KN-PA-10} to analyze limits of solutions with low-frequency initial conditions and large, high-frequency, random potentials.

The derivation of kinetic limits is much simplified in the setting of time-dependent random coefficients and does not require the aforementioned infinite Neumann series expansion. Assuming a Markovian structure for the random potential enables us to use a martingale method and a suitable perturbed test functions expansion. A limit theorem for one dimensional waves where such methods are used is given in \cite{papanicolaou1994functional} and more general ones in \cite{blankenship1978stability}. The same approach has been applied to Schr{\"o}dinger equation in different settings, see \cite{bal2002radiative,bal2002self,fannjiang2005self,gomez2011,olivier2012}.

The markovianity of the random coefficient simplifies derivations but is not necessary as was shown in \cite{fannjiang2005self}. In \cite{breteaux2011geometric}, a geometric approach is applied to more general initial data than in \cite{erdHos2000linear}. A renewal of the random field is used to get the appropriate mixing properties during the evolution.  In \cite{poupaud2003classical}, the authors consider random potentials that are correlated in time with finite range. They proved the semiclassical limit of Schr{\"o}dinger equation by pure PDE techniques.

For systems of equations such as the Dirac system or the linear hyperbolic systems considered in \cite{ryzhik1996transport}, no rigorous results have been established even for time-dependent potentials; see also \cite{bal2010kinetic} for a recent review on the derivation of limiting models for wave and particle propagation in random media. In this setting, the energy matrix has to be decomposed into different propagating and non-propagating modes, which behave quite differently in the high-frequency limit. We focus here on the Dirac equation although most of the derivations carried out in this paper extend to the framework of linear hyperbolic systems, which will be considered elsewhere. It should also be mentioned that the energy bands in the Dirac equation are degenerate, which is distinct from the scalar (Schr\"odinger) case.

The rest of the paper is organized as follows. We describe the problem setting and state our main results in Section \ref{sec:mainresults}. Next, we sketch the outline of the proof and discuss in detail the construction of test functions in Section \ref{sec:outlineproof}. In Section \ref{sec:proof}, we then prove the convergence of different modes. The case of slower fluctuations in time, which allow us to recover limiting kinetic models with elastic scattering, is briefly discussed in Section \ref{sec:slowerTime}. Conclusions and some further discussions are finally presented in Section \ref{sec:conclusion}.

\section{Main results}
\label{sec:mainresults}
In this section, we first describe our setting, next introduce the so-called Wigner transform as our main tool and then present our construction of the random electromagnetic field. We finally state the main results of the paper.

The Dirac equation in three dimensions of space reads
\begin{equation}
\eps\partial_t \Psi^\eps+P(x,\eps D_x)\Psi^\eps =0, \ \ x \in \R^3, t\in \R,
\label{eq:diracEq}
\end{equation}
with the differential operator $P(x,D)$  defined as:
\begin{equation}
P(x,\xi)=i\left( \sum_{k=1}^3 \gamma^0\gamma^k(\xi_k-eA_k(x))+m_0c\gamma^0-eA_0(x)I_4\right).
\end{equation}
Here $\Psi^\eps=\Psi^\eps(t,x)\in \CC^4$ is the wave function. $\eps=\hbar$ is our small parameter and stands for the Planck constant. $e$ is the unit charge, $m_0$ is the electron's rest mass, $c$ is the velocity of light, and $A_k(x)\in \R, k=0,1,2,3$ are the components of the prescribed electromagnetic field. In particular, $A_0$ is the electric potential and $(A_1,A_2,A_3)$ is the magnetic potential vector.

$\gamma^k\in \CC^{4\times4}, k=0,1,2,3$ are the $4\times 4$ Dirac matrices, which are closely related to the $2\times 2$ Pauli matrices. Their elements are $0,1,i$, and they satisfy
\begin{eqnarray*}
&&{\gamma^0}^*=\gamma^0, {\gamma^k}^*=-\gamma^k, k=1,2,3, (\gamma^0\gamma^k)^*=\gamma^0\gamma^k, \\
&&\gamma^m\gamma^n+\gamma^n\gamma^m=0, m\neq n, {\gamma^0}^2=I_4, {\gamma^k}^2=-I_4, k=1,2,3.
\end{eqnarray*}
In our proofs, the explicit form of the Dirac matrices does not play any special role.

The relativistic current density $J^\eps$ is a $4$ dimensional vector with elements $J_k^\eps$ given by
\begin{equation*}
J_k^\eps={\Psi^\eps}^*\gamma^0\gamma^k\Psi^\eps,
\end{equation*}
and the relativistic position density $n(t,x)$ is given by
\begin{equation}
n^\eps(t,x)=J_0^\eps(t,x)={\Psi^\eps}^*(t,x)\Psi^\eps(t,x).
\end{equation}

We investigate the limiting behavior of the solution to \eqref{eq:diracEq} as $\eps \to 0$. Given the conservation of $\int_{\R^d}n^\eps(t,x)dx$ and the fact that $n^\eps(t,x)$ does not admit a closed-formed equation,
the quantity we are interested in proves to be the Wigner transform of $\Psi^\eps$, which we introduce next. For more results on the Wigner transform, see \cite{gerard1997homogenization}.

\subsection{Wigner transform and pseudo-differential calculus}
We introduce the Wigner transform and some pseudo-differential calculus that is to be used.

The matrix-valued Wigner transform of two spatially-dependent $d$ dimensional vector fields $u(x)$ and $v(x)$ is defined as
\begin{equation}
W[u,v](x,\xi)=\int_{\R^d}\frac{1}{(2\pi)^d}e^{i\xi\cdot y}u(t,x-\frac{\eps y}{2})v^*(x+\frac{\eps y}{2})dy,
\end{equation}
where $v^*$ is the transposition and possible complex conjugate of $v$. It may be seen as the inverse Fourier transform of the two point correlation function of $u(x)$ and $v(x)$, where we define Fourier transform using the convention $\hat{f}(\xi)=\int_{\R^d} e^{-i\xi\cdot x}f(x) dx$. We check that:
\begin{equation}
(W[u,u](x,\xi))^*=W[u,u](x,\xi).
\end{equation}

We also verify that
\begin{equation}
\int_{\R^d} W[u,v](t,x,\xi)d\xi=(u v^*)(t,x),
\end{equation}
and this allows us to interpret the Wigner transform as the energy density in phase space, although the Wigner transform is positive only in the limit $\eps\to0$ \cite{gerard1997homogenization}.

We recall from \cite{bal2005kinetics} some simple results about pseudo-differential calculus that are needed in this paper.
\begin{proposition}
let $P(\eps D)$ be a matrix-valued pseudo-differential operator defined by
\begin{equation*}
P(\eps D)u(x)=\int_{\R^d} e^{ip\cdot x}P(i\eps p)\hat{u}(p)\frac{dp}{(2\pi)^d}.
\end{equation*}
Then we have
\begin{equation*}
W[P(\eps D)u,v](x,\xi)=P(i\xi+\frac{\eps D}{2})W[u,v](x,\xi),
\end{equation*}
and
\begin{equation*}
W[u,P(\eps D)v](x,\xi)=W[u,v](x,\xi)P^*(i\xi-\frac{\eps D}{2}),
\end{equation*}
where $W[u,v](x,\xi)P^*(i\xi-\eps D/2)$ is defined as the inverse Fourier transform $\F^{-1}_{p\to x}$ of the matrix $\hat{W}[u,v](p,\xi)P^*(i\xi-i\eps p/2)$.
\label{prop:pseudo1}
\end{proposition}
\begin{proposition}
let $V(x)$ be a real matrix-valued function, then we have
\begin{equation*}
W[V(\frac{x}{\eps})u,v](x,\xi)=\int_{\R^d}\frac{e^{ix\cdot p/\eps}}{(2\pi)^d}\hat{V}(p)W[u,v](x,\xi-\frac{p}{2})dp,
\end{equation*}
and
\begin{equation*}
W[u,V(\frac{x}{\eps})v](x,\xi)=\int_{\R^d}\frac{e^{ix\cdot p/\eps}}{(2\pi)^d}W[u,v](x,\xi+\frac{p}{2})\hat{V}^t(p)dp,
\end{equation*}
where $\hat{V}$ is the Fourier transform of $V$ component by component.
\label{prop:pseudo2}
\end{proposition}
\begin{remark}
Throughout the paper, we assume that spatial dimension $d=3$ although we occasionally use ``$d$" for expressions that hold independent of dimension.
\end{remark}

\subsection{The random field}
The $A_k(x)$ appearing in \eqref{eq:diracEq} are components of a random electromagnetic field, which we assume to be time-dependent and have mean zero. The non mean zero case can be handled similarly.

We follow the same construction of the random field as in \cite{bal2002radiative,bal2002self}. $\V$ is the set of measures of bounded total variation with support inside a ball $B_L={|p|\leq M}$:
\begin{equation*}
\V=\{\hat{V}=(\hat{V}_0,\hat{V}_1,\hat{V}_2,\hat{V}_3): \int_{\R^d}|d\hat{V}_k|\leq C, \ \ \mbox{supp}\ \ \hat{V}_k\subset B_L, \hat{V}_k(p)=\hat{V}_k^*(-p),k=0,1,2,3\}.
\end{equation*}
Let $\tilde{A}(t,p)=(\tilde{A}_0(t,p)dp,\tilde{A}_1(t,p)dp,\tilde{A}_2(t,p)dp,\tilde{A}_3(t,p)dp)$ be a mean-zero Markov process on $\V$ with generator $\Q$. The time-dependent random field $A_k(t,x)$ is given by
\begin{equation}
A_k(t,x)=\int_{\R^d} \frac{dp\tilde{A}_k(t,p)}{(2\pi)^d}e^{ip\cdot x}
\label{eq:defVtx}
\end{equation}
 and is real and uniformly bounded. We assume $A_k(t,x)$ is stationary in $t$ and $x$ and the correlation functions are defined by
 \begin{equation*}
 R_{mn}(t,x):=\E\{A_m(t+s,x+y)A_n(s,y)\}, m,n=0,1,2,3.
 \end{equation*}
  Furthermore, we have
 \begin{equation*}
 \E\{\tilde{A}_m(t+s,p)dp\tilde{A}_n(s,q)dq\}=(2\pi)^d\tilde{R}_{mn}(t,p)\delta(p+q)dpdq,
 \end{equation*}
where the power spectrum $\tilde{R}_{mn}$ is the Fourier transform of $R_{mn}(t,x)$ in $x$:
\begin{equation*}
\tilde{R}_{mn}(t,p)=\int_{\R^d} R_{mn}(t,x)e^{-ix\cdot p}dx.
\end{equation*}
For simplicity we assume that $\tilde{R}_{mn}(t,p)\in \mathcal{S}(\R\times\R^d)$ and define the space-time Fourier transform $\hat{R}_{mn}(\omega,p)$ as
\begin{equation*}
\hat{R}_{mn}(\omega,p )=\int_\R \tilde{R}_{mn}(t,p)e^{-i\omega t}dt.
\end{equation*}

We assume that the generator $\Q$ is a bounded operator on $L^\infty (\V)$ with a unique invariant measure $\pi(\hat{V})$, i.e., $\Q^*\pi=0$, and there exists $\alpha>0$ such that if $\langle g,\pi\rangle=0$, then
\begin{equation}
\|e^{t\Q}g\|_{L^\infty (\V)}\leq C\|g\|_{L^\infty(\V)}e^{-\alpha t}.
\label{eq:speGap}
\end{equation}
The simplest example of a generator with gap in the spectrum and invariant measure $\pi$ is a jump process on $\V$ where
\begin{equation*}
\Q g(\hat{V})=\int_\V g(\hat{V}_1)d\pi (\hat{V}_1)-g(\hat{V}), \int_\V d\pi (\hat{V})=1.
\end{equation*}
Given \eqref{eq:speGap}, the Fredholm alternative holds for the Poisson equation \begin{equation*}
\Q f=g\end{equation*} provided that $\langle g,\pi\rangle=0$. It has a unique solution $f$ with $\langle f,\pi\rangle =0$ and $\|f\|_{L^\infty(\V)}\leq C\|g\|_{L^\infty(\V)}$. The solution is given explicitly by
\begin{equation*}
f(\hat{V})=-\int_0^\infty dr e^{r\Q}g(\hat{V}),\end{equation*}
and the integral converges absolutely because of \eqref{eq:speGap}.

\subsection{Main theorem}
Before stating the main results we first derive the equation satisfied by the Wigner transform.

Recall the equation
\begin{equation*}
\eps\partial_t \Psi^\eps+P(x,\eps D_x)\Psi^\eps =0.
\end{equation*}
We replace $A_k(t,x)$ with $\sqrt{\eps}A_k(\frac{t}{\eps},\frac{x}{\eps})$ in the weak coupling limit, so the equation for $\Psi^\eps$ becomes
\begin{equation}
\eps \partial_t \Psi^\eps+
i\left( \sum_{k=1}^3 \gamma^0\gamma^k(\eps D_k-\sqrt{\eps}eA_k(\frac{t}{\eps},\frac{x}{\eps}))+m_0c\gamma^0-\sqrt{\eps}eA_0(\frac{t}{\eps},\frac{x}{\eps})I_4\right)\Psi^\eps=0.
\label{eq:diracEqnew}
\end{equation}
Define $W_\eps(t,x,\xi)=W[\Psi^\eps(t,.),\Psi^\eps(t,.)](x,\xi)$, and by Proposition \ref{prop:pseudo1}, \ref{prop:pseudo2} we can derive the equation satisfied by $W_\eps$
\begin{equation}
\begin{aligned}
\eps \partial_t W_\eps+&\sum_{k=1}^3 c\gamma^0\gamma^k P_k(i\xi+\frac{\eps D}{2}) W_\eps +\sum_{k=1}^3 W_\eps P_k^*(i\xi-\frac{\eps D}{2}) c\gamma^0\gamma^k+im_0c^2(\gamma^0 W_\eps-W_\eps \gamma^0)\\
-&ie\sqrt{\eps}\sum_{k=0}^3\gamma^0\gamma^k\K_\eps^k W_\eps+ie\sqrt{\eps}\sum_{k=0}^3\tilde{\K}_\eps^k W_\eps\gamma^0\gamma^k=0,
\label{eq:WignerEq}
\end{aligned}
\end{equation}
where the symbol $P_k(\xi)=\xi_k$, and the operators $\K_\eps^k, \tilde{\K}_\eps^k$ are defined as
\begin{eqnarray*}
\K_{\eps}^k\lambda(t,x,\xi)=\int_{\R^d}\frac{dp\tilde{A}_k(\frac{t}{\eps},p)}{(2\pi)^d}e^{ix\cdot p/\eps}\lambda(t,x,\xi-\frac{p}{2}),\\
\tilde{\K}_{\eps}^k\lambda(t,x,\xi)=\int_{\R^d}\frac{dp\tilde{A}_k(\frac{t}{\eps},p)}{(2\pi)^d}e^{ix\cdot p/\eps}\lambda(t,x,\xi+\frac{p}{2}).
\end{eqnarray*}

We define the $L^2$ inner product for matrix valued functions as follows:
\begin{equation}
\langle F(y),G(y)\rangle=\int\Tr(F^*G)(y)dy.
\end{equation}

The Dirac equation \eqref{eq:diracEqnew} preserves the $L^2$ norm of $\Psi^\eps$, and by the property of Fourier transform, we have $\forall t\in [0,T]$,
\begin{equation}
\|W_\eps(t)\|_{L^2(\R^{2d})}=(2\pi\eps)^{-d/2}\|\Psi^\eps(t)\|_{L^2(\R^d)}^2=(2\pi\eps)^{-d/2}\|\Psi^\eps(0)\|_{L^2(\R^d)}^2.
\end{equation}
Thus by choosing appropriate initial condition $\Psi^\eps(0)$ such that $\|\Psi^\eps(0)\|_{L^2(\R^d)}^2$ is of order $\eps^{d/2}$, we have $\|W_\eps(t)\|_{L^2(\R^{2d})}$ uniformly bounded for $t\in [0,T]$ and $\tilde{A}\in \V$. Moreover, we assume that $W_\eps(0,x,\xi)$ converges weakly as $\eps \to 0$ to $W_0(0,x,\xi)\in L^2(\R^{2d})$ without restriction to a subsequence.

Expanding $W_\eps=W_0+\sqrt{\eps}W_{1,\eps}+\eps W_{2,\eps}+\ldots$, dividing \eqref{eq:WignerEq} by $\eps$ and expanding in $\eps$, by setting the $1/\eps$ order term to be zero, we obtain that
\begin{equation*}
\sum_{k=1}^3 c\gamma^0\gamma^kP_k(i\xi)W_0+W_0\sum_{k=1}^3 P_k^*(i\xi)c\gamma^0\gamma^k+im_0c^2(\gamma^0W_0-W_0\gamma^0)=0.
\end{equation*}
This enables us to define the dispersion matrix
\begin{equation}
Q(\xi)=\sum_{k=1}^3 \gamma^0\gamma^k \xi_k+m_0c\gamma^0,
\label{eq:DisperMat}
\end{equation}
and we expect the limit to satisfy that
\begin{equation}
Q(\xi)W_0(t,x,\xi)=W_0(t,x,\xi)Q(\xi).
\label{eq:dispersionrelation}
\end{equation}

By the property of $\gamma^0\gamma^k$, we know $Q$ is Hermitian, and $Q^2=(m_0^2c^2+|\xi|^2)I_4$. Let
\begin{equation}
\lambda_\pm(\xi)=\pm\sqrt{m_0^2c^2+|\xi|^2},
\end{equation}
be the eigenvalues of $Q$ corresponding to the energy levels of electrons and positrons, respectively.  The orthonormal eigenvectors are
\begin{eqnarray*}
x_1(\xi)=\left(\begin{array}{c} 0 \\ \frac{m_0c}{\sqrt{2\lambda_+(\lambda_+-\xi_3)}} \\ \frac{\xi_1-i\xi_2}{\sqrt{2\lambda_+(\lambda_+-\xi_3)}} \\ \sqrt{\frac{\lambda_+-\xi_3}{2\lambda_+}} \end{array} \right),\qquad
x_2(\xi)=\left(\begin{array}{c} \sqrt{\frac{\lambda_+-\xi_3}{2\lambda_+}} \\ -\frac{\xi_1+i\xi_2}{\sqrt{2\lambda_+(\lambda_+-\xi_3)}} \\ \frac{m_0c}{\sqrt{2\lambda_+(\lambda_+-\xi_3)}} \\ 0 \end{array}\right),\\
y_1(\xi)=\left(\begin{array}{c} \frac{m_0c}{\sqrt{2\lambda_+(\lambda_+-\xi_3)}} \\ 0 \\ -\sqrt{\frac{\lambda_+-\xi_3}{2\lambda_+}} \\ \frac{\xi_1+i\xi_2}{\sqrt{2\lambda_+(\lambda_+-\xi_3)}}\end{array}\right),\qquad
y_2(\xi)=\left(\begin{array}{c} \frac{\xi_1-i\xi_2}{\sqrt{2\lambda_+(\lambda_+-\xi_3)}} \\ \sqrt{\frac{\lambda_+-\xi_3}{2\lambda_+}} \\ 0 \\ -\frac{m_0c}{\sqrt{2\lambda_+(\lambda_+-\xi_3)}}\end{array}\right),
\end{eqnarray*}
where $x_1,x_2$ correspond to $\lambda_+$ and $y_1,y_2$ correspond to $\lambda_-$.
\begin{remark}
The explicit form of the eigenvectors do not affect our results.
\end{remark}
Because $\{x_1,x_2,y_1,y_2\}$ is a basis in $\CC^4$, we can decompose $W_\eps$ as
\begin{equation}
W_\eps=\sum_{i,j=1,2} a_{ij}^\eps x_ix_j^*+\sum_{i,j=1,2} b_{ij}^\eps y_iy_j^*+\sum_{i,j=1,2} c_{ij}^\eps x_iy_j^*+\sum_{i,j=1,2} d_{ij}^\eps y_ix_j^*.
\end{equation}
On one hand, we have $W_\eps^*=W_\eps$, so
\begin{eqnarray*}
&&a_{ii}^\eps\in \R, i=1,2 \ \ a_{12}^\eps=\overline{a_{21}^\eps},\\
&&b_{ii}^\eps\in \R, i=1,2 \ \ b_{12}^\eps=\overline{b_{21}^\eps}.
\end{eqnarray*}
On the other hand, since $W_\eps$ is uniformly bounded in $L^2$, we have that all the coefficients $a_{ij}^\eps, b_{ij}^\eps, c_{ij}^\eps, d_{ij}^\eps$ belong to $\C([0,T];L^2(\R^{2d}))$. The dispersion relation \eqref{eq:dispersionrelation} implies that $c_{ij}^\eps,d_{ij}^\eps$ should converge to zero as $\eps \to 0$. Thus we write the limit
\begin{equation}
W_0=\sum_{i,j=1,2} a_{ij}x_ix_j^*+\sum_{i,j=1,2}b_{ij} y_iy_j^*.
\end{equation}
We denote by $\Pi_\pm$ the orthogonal projection of $\CC^4$ on the eigenspace associated to $\lambda_\pm$, and they are
\begin{eqnarray*}
&&\Pi_+=x_1x_1^*+x_2x_2^*=\frac{1}{2}(I_4+\frac{Q}{\lambda_+}), \\
&&\Pi_-=y_1y_1^*+y_2y_2^*=\frac{1}{2}(I_4-\frac{Q}{\lambda_+}).
\end{eqnarray*}
Define $\alpha_+^\eps=\Tr(\Pi_+W_\eps)=a_{11}^\eps+a_{22}^\eps$, $\alpha_-^\eps=\Tr(\Pi_-W_\eps)=b_{11}^\eps+b_{22}^\eps$. We check that
\begin{equation*}
\int_{\R^d} (\alpha_+^\eps+\alpha_-^\eps)(t,x,\xi)d\xi
=n^\eps(t,x).
\end{equation*}
Therefore we can interpret $\alpha_+^\eps$ and $\alpha_-^\eps$ as energy densities in the phase space corresponding to electrons and positrons, respectively. They are the physical quantities we are interested in.

Define $L^2_M(\R^{2d})=\{f|\|f\|_{L^2(\R^{2d})}\leq M\}$. The space we use for $(\alpha_+^\eps(t), \alpha_-^\eps(t))$ is
\begin{equation*}
H_M=L^2_M(\R^{2d})\bigoplus L^2_M(\R^{2d})=\{X=(X_1,X_2)|X_1,X_2\in L^2_M(\R^{2d})\}
\end{equation*}
endowed with the metric
\begin{equation*}
d_H(X,Y)=\sum_{n=1}^\infty \left(\frac{|\langle X_1-Y_1,e_n\rangle|}{2^n}
+\frac{|\langle X_2-Y_2,e_n\rangle|}{2^n}\right),
\end{equation*}
where $\{e_n\}$ is orthonormal basis of $L^2_M(\R^{2d})$. It is straightforward to check that $d_H$ induces the weak topology on $L^2_M(\R^{2d})\bigoplus L^2_M(\R^{2d})$ and $(H_M,d_H)$ is a complete separable metric space. We have the following tightness criteria about process taking value in $H_M$.
\begin{proposition}
Suppose $\Pb_\eps$ is the family of probability measures induced by $(\alpha_+^\eps,\alpha_-^\eps)$ on $\C([0,T];H_M)$, if for any $f\in \C^1([0,T];\mathcal{S}(\R^{2d}))$, the process $\langle f, \alpha_\pm^\eps\rangle\in \C([0,T];\R)$ is tight, then $\Pb_\eps$ is tight.
\label{prop:tightcondition}
\end{proposition}

Now we can state our main theorem.
\begin{theorem}
Suppose $W_\eps$ solves \eqref{eq:WignerEq} with an initial condition $W_\eps(0,x,\xi)$ that is uniformly bounded in $L^2(\R^{2d})$ and converges weakly in $L^2(\R^{2d})$ to a fixed, deterministic function $W_0(0,x,\xi)\in L^2(\R^{2d})$ as $\eps \to 0$. We decompose $W_\eps$ as
\begin{equation*}
W_\eps=\sum_{i,j=1,2} a_{ij}^\eps x_ix_j^*+\sum_{i,j=1,2} b_{ij}^\eps y_iy_j^*+\sum_{i,j=1,2} c_{ij}^\eps x_iy_j^*+\sum_{i,j=1,2} d_{ij}^\eps y_ix_j^*.
\end{equation*}
Then for the cross modes $c_{ij}^\eps, d_{ij}^\eps$ we have
\begin{equation}
\sum_{i,j=1,2}\left(|\int_0^T\int_{\R^{2d}}c_{ij}^\eps(t,x,\xi)f(t,x,\xi)dxd\xi dt|
+|\int_0^T\int_{\R^{2d}}d_{ij}^\eps(t,x,\xi)f(t,x,\xi)dxd\xi dt|\right)
\leq C_{f,T}\sqrt{\eps}
\end{equation}
 for any function $f\in \C^1([0,T];\mathcal{S}(\R^{2d}))$ almost surely. For the propagating modes, define $\alpha_+^\eps=a_{11}^\eps+a_{22}^\eps, \,\alpha_-^\eps=b_{11}^\eps+b_{22}^\eps$ and suppose $\Pb_\eps$ is the family of probability measures induced by $(\alpha_+^\eps,\alpha_-^\eps)$ on $\C([0,T];H_M)$. Then, as $\eps \to 0$, $\Pb_\eps$ converges weakly to the probability measure $\Pb=\delta_{(\alpha_+,\alpha_-)}$, where $(\alpha_+,\alpha_-)$ is the unique deterministic solution to the following transport equation system
\begin{equation}
\begin{aligned}
&\partial_t \alpha_++\frac{c\xi\cdot \nabla_x \alpha_+}{\lambda_+(\xi)}=\mathcal{T}(\alpha_+,\alpha_-),\\
&\partial_t \alpha_-+\frac{c\xi\cdot \nabla_x \alpha_-}{\lambda_-(\xi)}=\mathcal{T}(\alpha_-,\alpha_+).
\label{eq:tranEqsys}
\end{aligned}
\end{equation}
The initial conditions are given by $\alpha_\pm(0,x,\xi)=\Tr(\Pi_\pm W_0(0,x,\xi))$, and the scattering operator $\mathcal{T}$ is defined as
\begin{equation*}
\begin{aligned}
\mathcal{T}(\alpha_+,\alpha_-)
=&\frac{e^2}{(2\pi)^d}\int_{\R^d}(\alpha_+(q)-\alpha_+(\xi))\sum_{k=0}^3\omega_k(\xi,q)\hat{R}_{kk}(c\lambda_+(q)-c\lambda_+(\xi),q-\xi)dq\\
+&\frac{e^2}{(2\pi)^d}\int_{\R^d}(\alpha_-(q)-\alpha_+(\xi))\sum_{k=0}^3\tilde{\omega}_k(\xi,q)\hat{R}_{kk}(c\lambda_+(q)+c\lambda_+(\xi),q-\xi)dq,
\end{aligned}
\end{equation*}
where we have $\omega_k(\xi,q)+\tilde{\omega}_k(\xi,q)=1$ and
\begin{eqnarray*}
\omega_0(\xi,q)=\frac{\lambda_+(q)\lambda_+(\xi)+\xi_1q_1+\xi_2q_2+\xi_3q_3+m_0^2c^2}{2\lambda_+(q)\lambda_+(\xi)},\\
\omega_1(\xi,q)=\frac{\lambda_+(q)\lambda_+(\xi)+\xi_1q_1-\xi_2q_2-\xi_3q_3-m_0^2c^2}{2\lambda_+(q)\lambda_+(\xi)},\\
\omega_2(\xi,q)=\frac{\lambda_+(q)\lambda_+(\xi)-\xi_1q_1+\xi_2q_2-\xi_3q_3-m_0^2c^2}{2\lambda_+(q)\lambda_+(\xi)},\\
\omega_3(\xi,q)=\frac{\lambda_+(q)\lambda_+(\xi)-\xi_1q_1-\xi_2q_2+\xi_3q_3-m_0^2c^2}{2\lambda_+(q)\lambda_+(\xi)}.
\end{eqnarray*}
\label{theo:maintheorem}
\end{theorem}
\begin{remark}
We have assumed that $\tilde{R}_{mn}=0$ when $m\neq n$. Similar results could be obtained for the general case with more complicated expressions we do not reproduce here.
\end{remark}
\begin{remark}
We see from the structure of the scattering operator that due to the temporal regularization,  the energy $\sqrt{m_0^2c^2+|\xi|^2}$ is no longer conserved. Scattering is inelastic and we observe a coupling between the propagating modes $\alpha_+$ and $\alpha_-$.
\end{remark}
\begin{remark}
The uniqueness of solutions to the transport equation system comes from the fact that $\partial_t ( \langle \alpha_+,\alpha_+\rangle +\langle \alpha_-,\alpha_-\rangle)\leq 0$ under the dynamics of \eqref{eq:tranEqsys}. To see this, we only have to note that
\begin{equation*}
\langle \alpha_\pm, \frac{c\xi\cdot \nabla_x\alpha_\pm}{\lambda_\pm(\xi)}\rangle =0,
\end{equation*}
and
\begin{equation*}
\begin{aligned}
&\langle \alpha_+, \mathcal{T}(\alpha_+,\alpha_-)\rangle+\langle \alpha_-,\mathcal{T}(\alpha_-,\alpha_+)\rangle\\
=&-\frac{e^2}{2(2\pi)^d}\int_{\R^{3d}}\left[ (\alpha_+(q)-\alpha_+(\xi))^2+(\alpha_-(q)-\alpha_-(\xi))^2\right]
\sum_{k=0}^3\omega_k(\xi,q)\hat{R}_{kk}(c\lambda_+(q)-c\lambda_+(\xi),q-\xi)dqd\xi dx\\
&-\frac{e^2}{(2\pi)^d}\int_{\R^{3d}}(\alpha_+(\xi)-\alpha_-(q))^2
\sum_{k=0}^3\tilde{\omega}_k(\xi,q)\hat{R}_{kk}(c\lambda_+(q)+c\lambda_+(\xi),q-\xi)dqd\xi
dx\leq 0.
\end{aligned}
\end{equation*}
\end{remark}

\subsection{Comments on the limiting equation}

Using the same approach, we can more generally show that the matrix-valued process $(\mathbf{A}^\eps(t),\mathbf{B}^\eps(t))$ has the weak limit $(\mathbf{A}(t),\mathbf{B}(t))$ satisfying a transport equation system, where we have defined $\mathbf{A}^\eps(t)=(a_{ij}^\eps(t))$,$\mathbf{B}^\eps(t)=(b_{ij}^\eps(t))$ and $\mathbf{A}(t)=(a_{ij}(t))$,$\mathbf{B}(t)=(b_{ij}(t))$. In other words, we have a limiting transport equation system for $a_{ij}^\eps, b_{ij}^\eps$.

Using the operators defined in \eqref{eq:limitA}, we formally have
\begin{equation}
\partial_t W_0=\A^* W_0,
\label{eq:limitW0}
\end{equation}
where $W_0=\sum_{i,j=1,2}a_{ij}x_ix_j^*+\sum_{i,j=1,2}b_{ij}y_iy_j^*$. To derive the equation satisfied by $a_{ij}$ or $b_{ij}$, we only have to write \eqref{eq:limitW0} as $\partial_t x_i^*W_0x_j=x_i^*\A^*W_0x_j$ or $\partial_t y_i^*W_0y_j=y_i^*\A^*W_0y_j$ and compute $x_i^*\A^*W_0x_j, y_i^*\A^*W_0y_j$ respectively. It should be mentioned that $W_0$ does not satisfy \eqref{eq:limitW0} because the dynamics would inevitably generate modes of $x_iy_j^*, y_ix_j^*$.

Therefore, we have the limiting coefficient matrices $\mathbf{A}(t)$ and $\mathbf{B}(t)$ satisfying the following transport equation system:
\begin{equation}
\partial_t \mathbf{A}(t)=\left(
\begin{array}{cc}
x_1^*\A^*W_0x_1 & x_1^*\A^*W_0x_2 \\
x_2^*\A^*W_0x_1 & x_2^*\A^*W_0x_2
\end{array}\right),\quad
\partial_t \mathbf{B}(t)=\left(
\begin{array}{cc}
y_1^*\A^*W_0y_1 & y_1^*\A^*W_0y_2 \\
y_2^*\A^*W_0y_1 & y_2^*\A^*W_0y_2
\end{array}\right).
\label{eq:bigtranEqsys}
\end{equation}
So $\alpha_+=a_{11}+a_{22}$ and $\alpha_-=b_{11}+b_{22}$ satisfy
\begin{eqnarray*}
&&\partial_t\alpha_+=x_1^*\A^*W_0x_1+x_2^*\A^*W_0x_2,\\
&&\partial_t\alpha_-=y_1^*\A^*W_0y_1+y_2^*\A^*W_0y_2.
\end{eqnarray*}
If we calculate these expressions explicitly, we recover \eqref{eq:tranEqsys} in Theorem \ref{theo:maintheorem}.

From \eqref{eq:tranEqsys} we see that the equation of $\alpha_+=a_{11}+a_{22}$ does not involve $a_{ij},b_{ij}$ when $i\neq j$. This could be illustrated as follows.

First of all, we write the equations satisfied by $a_{11}$ and $a_{22}$ explicitly.
\begin{equation}
\begin{aligned}
\left(\frac{e^2}{(2\pi)^d}\right)^{-1}\left(\partial_t a_{11}+\frac{c\xi\cdot \nabla_x a_{11}}{\lambda_+(\xi)}\right)
=&\int_{\R^d} \sum_{k=0}^3 \sum_{i,j=1,2}a_{ij}(q)x_1^*(\xi)\gamma^0\gamma^kx_i(q)x_j^*(q)\gamma^0\gamma^kx_1(\xi)\hat{R}_{kk}^-dq\\
+&\int_{\R^d} \sum_{k=0}^3 \sum_{i,j=1,2}b_{ij}(q)x_1^*(\xi)\gamma^0\gamma^ky_i(q)y_j^*(q)\gamma^0\gamma^kx_1(\xi)\hat{R}_{kk}^+dq\\
-&\int_{\R^d}\sum_{k=0}^3
a_{11}(\xi)\left[ |x_1^*(\xi)\gamma^0\gamma^kx_1(q)|^2+|x_1^*(\xi)\gamma^0\gamma^kx_2(q)|^2\right]\hat{R}_{kk}^-dq\\
-&\int_{\R^d}\sum_{k=0}^3
a_{11}(\xi)\left[ |x_1^*(\xi)\gamma^0\gamma^ky_1(q)|^2+|x_1^*(\xi)\gamma^0\gamma^ky_2(q)|^2\right]\hat{R}_{kk}^+dq,
\label{eq:eqa11}
\end{aligned}
\end{equation}

\begin{equation}
\begin{aligned}
\left(\frac{e^2}{(2\pi)^d}\right)^{-1}\left(\partial_t a_{22}+\frac{c\xi\cdot \nabla_x a_{22}}{\lambda_+(\xi)}\right)=&\int_{\R^d} \sum_{k=0}^3 \sum_{i,j=1,2}a_{ij}(q)x_2^*(\xi)\gamma^0\gamma^kx_i(q)x_j^*(q)\gamma^0\gamma^kx_2(\xi)\hat{R}_{kk}^-dq\\
+&\int_{\R^d} \sum_{k=0}^3 \sum_{i,j=1,2}b_{ij}(q)x_2^*(\xi)\gamma^0\gamma^ky_i(q)y_j^*(q)\gamma^0\gamma^kx_2(\xi)\hat{R}_{kk}^+dq\\
-&\int_{\R^d}\sum_{k=0}^3
a_{22}(\xi)\left[ |x_2^*(\xi)\gamma^0\gamma^kx_1(q)|^2+|x_2^*(\xi)\gamma^0\gamma^kx_2(q)|^2\right]\hat{R}_{kk}^-dq\\
-&\int_{\R^d}\sum_{k=0}^3
a_{22}(\xi)\left[ |x_2^*(\xi)\gamma^0\gamma^ky_1(q)|^2+|x_2^*(\xi)\gamma^0\gamma^ky_2(q)|^2\right]\hat{R}_{kk}^+dq,
\label{eq:eqa22}
\end{aligned}
\end{equation}
where $\hat{R}_{kk}^-:=\hat{R}_{kk}(c\lambda_+(q)-c\lambda_+(\xi),q-\xi), \hat{R}_{kk}^+:=\hat{R}_{kk}(c\lambda_+(q)+c\lambda_+(\xi),q-\xi)$. So the equations for $a_{11}$ and $a_{22}$ are indeed related to $a_{12},a_{21},b_{12},b_{21}$. However, if we sum \eqref{eq:eqa11} and \eqref{eq:eqa22}, with the observation that
\begin{equation*}
\gamma^0\gamma^kQ(\xi)\gamma^0\gamma^kQ(q)+
Q(q)\gamma^0\gamma^kQ(\xi)\gamma^0\gamma^k=c_kI_4
\end{equation*}
for some constant $c_k$, we have all terms concerning $a_{ij},b_{ij}$ with $i\neq j$ cancel out.

Take $a_{12}$ for example, we note that the coefficient of $a_{12}$ is
\begin{equation*}
\begin{aligned}
(I)=&x_1^*(\xi)\gamma^0\gamma^kx_1(q)x_2^*(q)\gamma^0\gamma^kx_1(\xi)+x_2^*(\xi)\gamma^0\gamma^kx_1(q)x_2^*(q)\gamma^0\gamma^kx_2(\xi)\\
=&x_2^*(q)\gamma^0\gamma^k(x_1(\xi)x_1^*(\xi)+x_2(\xi)x_2^*(\xi))\gamma^0\gamma^kx_1(q)\\
=&x_2^*(q)\gamma^0\gamma^k\frac{1}{2}(I_4+\frac{Q(\xi)}{\lambda_+(\xi)})\gamma^0\gamma^kx_1(q)=x_2^*(q)\frac{1}{2}(I_4+\gamma^0\gamma^k\frac{Q(\xi)}{\lambda_+(\xi)}\gamma^0\gamma^k)x_1(q)\\
=&x_2^*(q)Q(q)\gamma^0\gamma^kQ(\xi)\gamma^0\gamma^kx_1(q)/(2\lambda_+(q)\lambda_+(\xi))\\
=&x_2^*(q)\gamma^0\gamma^kQ(\xi)\gamma^0\gamma^kQ(q)x_1(q)/(2\lambda_+(q)\lambda_+(\xi)),
\end{aligned}
\end{equation*}
so
\begin{equation*}
2(I)=x_2^*(q)\left[\gamma^0\gamma^kQ(\xi)\gamma^0\gamma^kQ(q)+
Q(q)\gamma^0\gamma^kQ(\xi)\gamma^0\gamma^k\right]x_1(q)=0.
\end{equation*}
Similar expressions hold for $a_{21},b_{12},b_{21}$.

\section{Outline of the proof}
\label{sec:outlineproof}
The proof of cross modes converging to zero comes from the equation itself and the a priori $L^2$ boundedness of $W_\eps$. For the propagating modes $\alpha_\pm^\eps$, the basic idea is to use the Markovian property of $\tilde{A}_k(t,p)dp,\, k=0,1,2,3$. We construct appropriate test functions, then prove the approximating martingale property which will be specified later. Using the approximating martingale inequalities,  we can prove the uniqueness of the limit; together with the tightness of $(\alpha_+^\eps,\alpha_-^\eps)$, we finish the proof of the main theorem.

First, we define the operator $\Li$ and $\mathcal{A}$
\begin{equation*}
\begin{aligned}
\Li\lambda_0(x,\xi)=
&\frac{e^2}{(2\pi)^d}\sum_{m,n=0}^3\int_0^\infty \int_{\R^d}\gamma^0\gamma^m\tilde{R}_{mn}(r,q)e^{-rA_1(\xi-\frac{q}{2},-q)}\left[
\lambda_0(\xi-q)\gamma^0\gamma^n-\gamma^0\gamma^n\lambda_0(\xi)\right]e^{-rA_2(\xi-\frac{q}{2},-q)}dqdr\\
-&\frac{e^2}{(2\pi)^d}
\sum_{m,n=0}^3\int_0^\infty\int_{\R^d}e^{-rA_1(\xi-\frac{q}{2},q)}\left[\lambda_0(\xi)\gamma^0\gamma^m-\gamma^0\gamma^m\lambda_0(\xi-q)\right]
e^{-rA_2(\xi-\frac{q}{2},q)}\gamma^0\gamma^n\tilde{R}_{mn}(r,q)dqdr,
\end{aligned}
\end{equation*}
\begin{equation}
\mathcal{A}\lambda_0=\frac{1}{2}\sum_{k=1}^3(c\gamma^0\gamma^kD_k\lambda_0+\frac{1}{2}D_k\lambda_0 c\gamma^0\gamma^k)+\Li\lambda_0,
\label{eq:limitA}
\end{equation}
where
\begin{equation}
\begin{array}{l}
\displaystyle A_1(\xi,p)=\sum_{k=1}^3c\gamma^0\gamma^kP_k^*(i\xi+\frac{ip}{2})-im_0c^2\gamma^0,\\
\displaystyle A_2(\xi,p)=\sum_{k=1}^3 P_k(i\xi-\frac{ip}{2})c\gamma^0\gamma^k+im_0c^2\gamma^0.
\end{array}
\label{eq:A1A2}
\end{equation}

We will see later that for any test function $\lambda_0$ such that $Q\lambda_0=\lambda_0Q$, we have $\langle \lambda_0,W_0\rangle$ formally satisfies the equation
\begin{equation*}
\partial_t\langle \lambda_0,W_0\rangle=\langle (\partial_t+\mathcal{A})\lambda_0,W_0\rangle,
\end{equation*}
and this motivates us to define the first approximating martingale functional \\$\G^1_{\lambda_0}:\C([0,T];L^2(\R^{2d}))\to \C([0,T])$ as
\begin{equation}
\G^1_{\lambda_0}[W](t)=\langle \lambda_0,W\rangle(t)-\int_0^t ds\langle (\partial_t+\mathcal{A})\lambda_0,W\rangle(s).
\end{equation}
Assume $\tilde{\Pb}_\eps$ is the probability measure induced by $W_\eps$ on $\C([0,T];L^2(\R^{2d}))$. Our first goal is to show that under $\tilde{\Pb}_\eps$, $\G^1_{\lambda_0}[W](t)$ is an approximating martingale. More precisely, we will show
\begin{equation}
|\E^{\tilde{\Pb}_\eps}\{\G^1_{\lambda_0}[W](t)|\F_s\}-\G^1_{\lambda_0}[W](s)|\leq C_{\lambda_0,T}\sqrt{\eps},
\label{eq:1stapproMartin}
\end{equation}
uniformly for all $W\in \C([0,T];L^2(\R^{2d}))$ and $0\leq s<t\leq T$.

 Supposing $(\alpha_+^\eps,\alpha_-^\eps)\Rightarrow (\alpha_+,\alpha_-)$ and choosing appropriate test functions $\lambda_0$ in \eqref{eq:1stapproMartin}, we obtain from the above bound that  $(\E\{\alpha_+\},\E\{\alpha_-\})$ satisfies the transport equation system \eqref{eq:tranEqsys} by setting $s=0$ and $\eps \to 0$.

On the other hand, for any test function $F_0$ and $G_0$ satisfying the dispersion relation, i.e., $QF_0=F_0Q, QG_0=G_0Q$, let $\mu_0=F_0\otimes G_0$  and define the second approximating martingale functional $\G^2_{\mu_0}:\C([0,T];L^2(\R^{2d}))\to \C([0,T])$ as
\begin{equation}
\G^2_{\mu_0}[W](t)=\langle  F_0\otimes G_0, W\otimes W\rangle(t)-\int_0^tds\langle
F_0\otimes \left[(\partial_t+\mathcal{A})G_0\right]+\left[(\partial_t+\mathcal{A})F_0\right]\otimes G_0,W\otimes W\rangle(s).
\end{equation}
Our second goal is to show that under $\tilde{\Pb}_\eps$, $\G^2_{\mu_0}[W](t)$ is an approximating martingale, i.e.,
\begin{equation}
|\E^{\tilde{\Pb}_\eps}\{\G^2_{\mu_0}[W](t)|\F_s\}-\G^2_{\mu_0}[W](s)|\leq C_{\mu_0,T}\sqrt{\eps},
\label{eq:2ndapproMartin}
\end{equation}
uniformly for all $W\in \C([0,T];L^2(\R^{2d}))$ and $0\leq s<t\leq T$.

Similarly, by choosing $s=0$ and letting $\eps\to 0$, we can show that $\E\{\alpha_+\otimes \alpha_+\}=\E\{\alpha_+\}\otimes\E\{\alpha_+\}$ and $\E\{\alpha_-\otimes \alpha_-\}=\E\{\alpha_-\}\otimes\E\{\alpha_-\}$, so $(\alpha_+,\alpha_-)$ is deterministic hence the solution to the transport equation system \eqref{eq:tranEqsys}.

We next go through the construction of test functions.

\subsection{Construction of $\lambda_{1,\eps},\lambda_{2,\eps}$}
In order to obtain the approximating martingale inequalities, we have to consider the conditional expectation of functions $F(\hat{V},W)$. The only functions we are interested in are those of the form $F(\hat{V},W)=\langle \lambda(\hat{V}),W\rangle$. Given a function $F(\hat{V},W)$, if we denote by $\bar{\Pb}_\eps$ the probability measure induced by $(W_\eps(t),\tilde{A}(t/\eps))$ on the space $\C([0,T];L^2(\R^{2d}))\times\V$, the conditional expectation is defined as
\begin{equation*}
\E^{\bar{\Pb}_\eps}_{W,\hat{V},t}\{F(\hat{V},W)\}(\tau)
=\E^{\bar{\Pb}_\eps}\{F(\tilde{A}(\tau),W(\tau))|\tilde{A}(t)=\hat{V},W(t)
=W\}, \tau\geq t.
\end{equation*}

The weak form of the infinitesimal generator of the Markov process induced by $(W_\eps(t),\tilde{A}(t/\eps))$  is
\begin{equation}
\frac{d}{dh}\E^{\bar{\Pb}_\eps}_{W,\hat{V},t}\{
\langle \lambda(\hat{V}),W\rangle\}(t+h)|_{h=0}
=\langle \frac{1}{\eps}\Q\lambda+\partial_t\lambda+\A_\eps\lambda,W\rangle,
 \label{eq:infiniGene}
 \end{equation}
 provided we write the equation \eqref{eq:WignerEq} as $\partial_t W_\eps=\A_\eps^* W_\eps$ with $\A_\eps^*$ the adjoint of $\A_\eps$.

We have
\begin{equation}
\begin{aligned}
\A_\eps\lambda=&-\frac{1}{\eps}\sum_{k=1}^3c\gamma^0\gamma^kP_k^*(i\xi+\frac{\eps D}{2})\lambda-\frac{1}{\eps}\sum_{k=1}^3\lambda P_k(i\xi-\frac{\eps D}{2})c\gamma^0\gamma^k+\frac{1}{\eps}im_0c^2\gamma^0\lambda-\frac{1}{\eps}im_0c^2\lambda\gamma^0\\
&-\frac{1}{\sqrt{\eps}}ie\sum_{k=0}^3\gamma^0\gamma^k\K_{\eps}^k\lambda+\frac{1}{\sqrt{\eps}}ie\sum_{k=0}^3\tilde{\K}_{\eps}^k\lambda\gamma^0\gamma^k.
\end{aligned}
\label{eq:adjointAeps}
\end{equation}

Let $\lambda_\eps=\lambda_0+\sqrt{\eps}\lambda_{1,\eps}+\eps\lambda_{2,\eps}$ with $\lambda_0$ satisfying $Q(\xi)\lambda_0(\xi)=\lambda_0(\xi)Q(\xi)$. Plugging this expressions into $(\frac{1}{\eps}Q+\partial_t+\A_\eps)\lambda_\eps$ and equating like powers of $\eps$, we find that the term of order $1/\eps$ equal to zero.

Considering the term of order $1/\sqrt{\eps}$, we introduce the fast variable $z=x/\eps$ and define $\lambda_{1,\eps}=\lambda_1(\tilde{A},t,x,x/\eps,\xi)$, where $\lambda_1=\lambda_1(\tilde{A},t,x,z,\xi)$ solves
\begin{equation}
\Q\hat{\lambda}_1-A_1\hat{\lambda}_1-\hat{\lambda}_1A_2=F_1,
\label{eq:lambda1Eq}
\end{equation}
with $\hat{\lambda}_1=\F_{z\to p}\lambda_1$, $A_1,A_2$ defined in \eqref{eq:A1A2} and
\begin{eqnarray*}
&&F_1=ie G_t(p)\lambda_0(\xi-\frac{p}{2})
-ie\lambda_0(\xi+\frac{p}{2})G_t(p),\\
&&G_t(p)=\sum_{k=0}^3\gamma^0\gamma^k\tilde{A}_k(\frac{t}{\eps},p).
\end{eqnarray*}
 The solution to \eqref{eq:lambda1Eq} is given by
\begin{equation}
\lambda_1=-\int_{\R^d}\frac{e^{iz\cdot p}}{(2\pi)^d}\int_0^\infty e^{r\Q}e^{-rA_1}F_1e^{-rA_2}drdp.
\label{eq:lambda1}
\end{equation}

Similarly, we can define $\lambda_{2,\eps}=\lambda_2(\tilde{A},t,x,x/\eps,\xi)$ with $\lambda_2=\lambda_2(\tilde{A},t,x,z,\xi)$ solving
\begin{equation}
\Q\hat{\lambda}_2-A_1\hat{\lambda}_2-\hat{\lambda}_2A_2=F_2+\Li\lambda_0(2\pi)^d\delta(p),
\label{eq:lambda2Eq}
\end{equation}
where \begin{equation*}
F_2=\frac{ie}{(2\pi)^d}\int_{\R^d}G_t(q)\hat{\lambda}_1(p-q,\xi-\frac{q}{2})dq
-\frac{ie}{(2\pi)^d}\int_{\R^d}\hat{\lambda}_1(p-q,\xi+\frac{q}{2})G_t(q)dq.
\end{equation*}
The solvability comes from the fact that $\E\{F_2\}+\Li\lambda_0(2\pi)^d\delta(p)=0$, and the solution is given by
\begin{equation}
\lambda_2=-\int_{\R^d}\frac{e^{iz\cdot p}}{(2\pi)^d}\int_0^\infty e^{r\Q}e^{-rA_1}(F_2+\Li\lambda_0(2\pi)^d\delta(p))e^{-rA_2}drdp.
\label{eq:lambda2}
\end{equation}

By \eqref{eq:infiniGene}, we know that
\begin{equation*}
\G_{\lambda_0}^{1,\eps}[W](t):=\langle \lambda_\eps,W\rangle(t)-\int_0^t ds\langle (\frac{1}{\eps}\Q+\partial_t+\A_\eps)\lambda_\eps,W\rangle(s)
\end{equation*}
is a $\bar{\Pb}_\eps$-martingale. With $\lambda_\eps=\lambda_0+\sqrt{\eps}\lambda_{1,\eps}+\eps\lambda_{2,\eps}$, we have
\begin{equation*}
\G_{\lambda_0}^{1,\eps}[W](t)=\G_{\lambda_0}^1[W](t)+\langle C_1,W\rangle(t)+\int_0^tds
\langle C_2+C_3+C_4,W\rangle(s),
\end{equation*}
where the correctors $C_i$ are
\begin{eqnarray*}
&&C_1=\sqrt{\eps}\lambda_{1,\eps}+\eps\lambda_{2,\eps},\\
&&C_2=-\partial_t(\sqrt{\eps}\lambda_{1,\eps}+\eps\lambda_{2,\eps}),\\
&&C_3=\sqrt{\eps}ie\sum_{k=0}^3\gamma^0\gamma^k\K_{\eps}^k\lambda_{2,\eps}
-\sqrt{\eps}ie\sum_{k=0}^3\tilde{\K}_{\eps}^k\lambda_{2,\eps}\gamma^0\gamma^k,
\end{eqnarray*}
\begin{equation*}
C_4=\frac{1}{\eps}\sum_{k=1}^3c\gamma^0\gamma^kP_k^*(\frac{\eps D_x}{2})(\sqrt{\eps}\lambda_{1,\eps}+\eps\lambda_{2,\eps})+\frac{1}{\eps}\sum_{k=1}^3(\sqrt{\eps}\lambda_{1,\eps}+\eps\lambda_{2,\eps})P_k(-\frac{\eps D_x}{2})c\gamma^0\gamma^k.
\end{equation*}
\begin{remark}
In $C_4$, the derivative $D_x$ is with respect to the slow variable.
\end{remark}
To prove that $\G_{\lambda_0}^1[W](t)$ is a $\tilde{\Pb}_\eps$-approximating martingale, we need to show that the following correctors are small:
\begin{equation*}
\|\langle \sum_{i=1}^4C_i,\sum_{i=1}^4 C_i\rangle(t)\|_{L^\infty(\V)}\leq C_{\lambda_0,T}\eps
\end{equation*}
uniformly in $t\in [0,T]$.

\subsection{Construction of $\mu_{1,\eps},\mu_{2,\eps}$}
First of all, we derive the equation satisfied by $W_\eps\otimes W_\eps(t,x_1,x_2,\xi_1,\xi_2)$. Define the $16\times 16$ matrices
\begin{equation*}
\Gamma^k=\left(\begin{array}{cccc}
\gamma^k_{11}I_4 & \gamma^k_{12}I_4 & \gamma^k_{13}I_4 & \gamma^k_{14}I_4\\
\gamma^k_{21}I_4 & \gamma^k_{22}I_4 & \gamma^k_{23}I_4 & \gamma^k_{24}I_4\\
\gamma^k_{31}I_4 & \gamma^k_{32}I_4 & \gamma^k_{33}I_4 & \gamma^k_{34}I_4\\
\gamma^k_{41}I_4 & \gamma^k_{42}I_4 & \gamma^k_{43}I_4 & \gamma^k_{44}I_4\end{array}\right),\qquad
\tilde{\Gamma}^k=
\left(\begin{array}{cccc}
\gamma^k & 0 & 0 & 0 \\
0 & \gamma^k & 0 & 0\\
0 & 0 & \gamma^k & 0\\
0 & 0 & 0 & \gamma^k\end{array}\right).
\end{equation*}
We have
\begin{equation}
\begin{aligned}
&\eps\partial_t W_\eps\otimes W_\eps+\sum_{k=1}^3 \left[c\Gamma^0\Gamma^kP_{k1}(i\xi+\frac{\eps D}{2})+c\tilde{\Gamma}^0\tilde{\Gamma}^kP_{k2}(i\xi+\frac{\eps D}{2})\right]W_\eps\otimes W_\eps+im_0c^2(\Gamma^0+\tilde{\Gamma}^0)W_\eps\otimes W_\eps\\
+&W_\eps\otimes W_\eps\sum_{k=1}^3 \left[c\Gamma^0\Gamma^kP_{k1}^*(i\xi-\frac{\eps D}{2})+c\tilde{\Gamma}_0\tilde{\Gamma}^kP_{k2}^*(i\xi-\frac{\eps D}{2})\right]-im_0c^2W_\eps\otimes W_\eps(\Gamma^0+\tilde{\Gamma}^0)\\
-&ie\sqrt{\eps}\sum_{k=0}^3(\Gamma^0\Gamma^k\K_{\eps 1}^k+\tilde{\Gamma}^0\tilde{\Gamma}^k\K_{\eps 2}^k)W_\eps\otimes W_\eps+ie\sqrt{\eps} W_\eps\otimes W_\eps \sum_{k=0}^3 (\Gamma^0\Gamma^k\tilde{\K}_{\eps 1}^k+\tilde{\Gamma}^0\tilde{\Gamma}^k\tilde{\K}_{\eps 2}^k)=0,
\label{eq:2ndWignerEq}
\end{aligned}
\end{equation}
where $P_{ki},P_{ki}^*, \K_{\eps i}^k, \tilde{\K}_{\eps i}^k$ are the corresponding operators with respect to $(x_i,\xi_i), i=1,2$.

The weak form of the infinitesimal generator of the Markov process $\langle \mu, W_\eps\otimes W_\eps\rangle$ is
 \begin{equation*}
 \frac{d}{dh}\E^{\bar{\Pb}_\eps}_{W,\hat{V},t}\{\langle \mu(\hat{V}),W\otimes W\rangle \}(t+h)|_{h=0}=\langle (\frac{1}{\eps}\Q+\partial_t+\mathcal{B}_\eps)\mu, W\otimes W\rangle
 \end{equation*}
 if we write \eqref{eq:2ndWignerEq} as $\partial_t (W_\eps\otimes W_\eps)=\mathcal{B}_\eps^* (W_\eps\otimes W_\eps)$. Define $\mu_\eps=\mu_0+\sqrt{\eps}\mu_{1,\eps}+\eps\mu_{2,\eps}$. Then
 \begin{equation*}
 \G_{\mu_0}^{2,\eps}[W](t):=\langle \mu_\eps, W\otimes W\rangle(t)-\int_0^t ds \langle (\frac{1}{\eps}\Q+\partial_t+\mathcal{B}_\eps)\mu_\eps, W\otimes W\rangle(s)
 \end{equation*}
 is a $\bar{\Pb}_\eps$-martingale.

 Consider $(\frac{1}{\eps}\Q+\partial_t+\B_\eps)\mu_\eps$, after expanding in $\eps$, the term of order $1/\eps$ is
 \begin{equation*}
 \begin{aligned}
 (I)=&-\frac{1}{\eps}\sum_{k=1}^3c\left[\Gamma^0\Gamma^kP_{k1}^*(i\xi)+\tilde{\Gamma}^0\tilde{\Gamma}^kP_{k2}^*(i\xi)\right]\mu_0
 -\frac{1}{\eps}\mu_0\sum_{k=1}^3c\left[\Gamma^0\Gamma^kP_{k1}(i\xi)+\tilde{\Gamma}^0\tilde{\Gamma}^kP_{k2}(i\xi)\right]\\
 &+\frac{1}{\eps}im_0c^2(\Gamma^0+\tilde{\Gamma}^0)\mu_0
 -\frac{1}{\eps}im_0c^2\mu_0(\Gamma^0+\tilde{\Gamma}^0).
 \end{aligned}
 \end{equation*}
If we choose $\mu_0=F_0\otimes G_0$ with $F_0,G_0$ both satisfying the dispersion relation, i.e.,
\begin{equation*}
Q(\xi)F_0(\xi)=F_0(\xi)Q(\xi), Q(\xi)G_0(\xi)=G_0(\xi)Q(\xi),
\end{equation*}
we can check that $(I)=0$. For the term of order $1/\sqrt{\eps}$, similarly we introduce fast variables $z_1=x_1/\eps, z_2=x_2/\eps$, and define \begin{equation*}
\mu_{1,\eps}(\tilde{A},t,x_1,x_2,\xi_1,\xi_2)=\mu_1(\tilde{A},t,x_1,\frac{x_1}{\eps},x_2,\frac{x_2}{\eps},\xi_1,\xi_2),
\end{equation*} with $\mu_1$ solving
\begin{equation}
\begin{aligned}
\Q\mu_1-&\sum_{k=1}^3c\left[\Gamma^0\Gamma^k P_{k1}^*(i\xi+\frac{D_z}{2})+\tilde{\Gamma}^0\tilde{\Gamma}^kP_{k2}^*(i\xi+\frac{D_z}{2})\right]\mu_1+im_0c^2(\Gamma^0+\tilde{\Gamma}^0)\mu_1\\
-&\mu_1\sum_{k=1}^3c\left[\Gamma^0\Gamma^kP_{k1}(i\xi-\frac{D_z}{2})+\tilde{\Gamma}^0\tilde{\Gamma}^kP_{k2}(i\xi-\frac{D_z}{2})\right]-im_0c^2\mu_1(\Gamma^0+\tilde{\Gamma}^0)\\
=& ie\sum_{k=0}^3(\Gamma^0\Gamma^k\K_{\eps 1}^k+\tilde{\Gamma}^0\tilde{\Gamma}^k\K_{\eps 2}^k)\mu_0-ie\mu_0\sum_{k=0}^3 (\Gamma^0\Gamma^k\tilde{\K}_{\eps 1}^k+\tilde{\Gamma}^0\tilde{\Gamma}^k\tilde{\K}_{\eps 2}^k).
\label{eq:mu1Eq}
\end{aligned}
\end{equation}
We can check that the solution $\mu_1=F_1\otimes G_0+F_0\otimes G_1$, where $F_1,G_1$ solve \eqref{eq:lambda1Eq} with $\lambda_0$ replaced by $F_0,G_0$ respectively.
\begin{remark}
In \eqref{eq:mu1Eq}, we replace $x_i/\eps$ by $z_i$ in the operators $\K_{\eps i}^k, \tilde{\K}_{\eps i}^k$.
\end{remark}
In the same way, we define
\begin{equation*}
\mu_{2,\eps}(\tilde{A},t,x_1,x_2,\xi_1,\xi_2)=\mu_2(\tilde{A},t,x_1,\frac{x_1}{\eps},x_2,\frac{x_2}{\eps},\xi_1,\xi_2),
\end{equation*}
and $\mu_2$ solves
\begin{equation}
\begin{aligned}
\Q\mu_2-&\sum_{k=1}^3c\left[\Gamma^0\Gamma^k P_{k1}^*(i\xi+\frac{D_z}{2})+\tilde{\Gamma}^0\tilde{\Gamma}^kP_{k2}^*(i\xi+\frac{D_z}{2})\right]\mu_2+im_0c^2(\Gamma^0+\tilde{\Gamma}^0)\mu_2\\
-&\mu_2\sum_{k=1}^3c\left[\Gamma^0\Gamma^kP_{k1}(i\xi-\frac{D_z}{2})+\tilde{\Gamma}^0\tilde{\Gamma}^kP_{k2}(i\xi-\frac{D_z}{2})\right]-im_0c^2\mu_2(\Gamma^0+\tilde{\Gamma}^0)\\
=& ie\sum_{k=0}^3(\Gamma^0\Gamma^k\K_{\eps 1}^k+\tilde{\Gamma}^0\tilde{\Gamma}^k\K_{\eps 2}^k)\mu_1-ie\mu_1\sum_{k=0}^3 (\Gamma^0\Gamma^k\tilde{\K}_{\eps 1}^k+\tilde{\Gamma}^0\tilde{\Gamma}^k\tilde{\K}_{\eps 2}^k)+\mathbf{L}\mu_0.
\label{eq:mu2Eq}
\end{aligned}
\end{equation}
We define $\mathbf{L}\mu_0$ such that the RHS of \eqref{eq:mu2Eq} has mean zero so the equation is solvable, i.e.,
\begin{equation*}
\mathbf{L}\mu_0=\E\{-ie\sum_{k=0}^3(\Gamma^0\Gamma^k\K_{\eps 1}^k+\tilde{\Gamma}^0\tilde{\Gamma}^k\K_{\eps 2}^k)\mu_1+ie\mu_1\sum_{k=0}^3 (\Gamma^0\Gamma^k\tilde{\K}_{\eps 1}^k+\tilde{\Gamma}^0\tilde{\Gamma}^k\tilde{\K}_{\eps 2}^k)\}.
\end{equation*}
We decompose $\mathbf{L}\mu_0$ into two parts, i.e., $\mathbf{L}\mu_0=\mathbf{L}_1\mu_0+\mathbf{L}_2\mu_0$, where
\begin{equation*}
\begin{aligned}
\mathbf{L}_1\mu_0=&ie\E\{-\sum_{k=0}^3\Gamma^0\Gamma^k\K_{\eps 1}^kF_1\otimes G_0+F_1\otimes G_0
\sum_{k=0}^3\Gamma^0\Gamma^k\tilde{\K}_{\eps 1}^k-\sum_{k=0}^3\tilde{\Gamma}^0\tilde{\Gamma}^k\K_{\eps 2}^kF_0\otimes G_1+F_0\otimes G_1\sum_{k=0}^3\tilde{\Gamma}^0\tilde{\Gamma}^k\tilde{\K}_{\eps 2}^k\}\\
=&(\Li F_0)\otimes G_0+F_0\otimes (\Li G_0),
\end{aligned}
\end{equation*}
and
\begin{equation*}
\mathbf{L}_2 \mu_0=ie\E\{-\sum_{k=0}^3\Gamma^0\Gamma^k\K_{\eps 1}^kF_0\otimes G_1+F_0\otimes G_1
\sum_{k=0}^3\Gamma^0\Gamma^k\tilde{\K}_{\eps 1}^k-\sum_{k=0}^3\tilde{\Gamma}^0\tilde{\Gamma}^k\K_{\eps 2}^kF_1\otimes G_0+F_1\otimes G_0\sum_{k=0}^3\tilde{\Gamma}^0\tilde{\Gamma}^k\tilde{\K}_{\eps 2}^k\}.
\end{equation*}
By this decomposition, we can write $\mu_2=F_0\otimes G_2+F_2\otimes G_0+\tilde{\mu}_2$, where $F_2,G_2$ solve \eqref{eq:lambda2Eq} with $\lambda_0$ replaced by $F_0,G_0$ and $\tilde{\mu}_2$ solves the following equation
\begin{equation*}
\begin{aligned}
 &\Q\tilde{\mu}_2-\sum_{k=1}^3c\left[\Gamma^0\Gamma^k P_{k1}^*(i\xi+\frac{D_z}{2})+\tilde{\Gamma}^0\tilde{\Gamma}^kP_{k2}^*(i\xi+\frac{D_z}{2})\right]\tilde{\mu}_2+im_0c^2(\Gamma^0+\tilde{\Gamma}^0)\tilde{\mu}_2\\
-&\tilde{\mu}_2\sum_{k=1}^3c\left[\Gamma^0\Gamma^kP_{k1}(i\xi-\frac{D_z}{2})+\tilde{\Gamma}^0\tilde{\Gamma}^kP_{k2}(i\xi-\frac{D_z}{2})\right]-im_0c^2\tilde{\mu}_2(\Gamma^0+\tilde{\Gamma}^0)\\
=& \mathbf{L}_2\mu_0+ie\E\{\sum_{k=0}^3\Gamma^0\Gamma^k\K_{\eps 1}^kF_0\otimes G_1-F_0\otimes G_1
\sum_{k=0}^3\Gamma^0\Gamma^k\tilde{\K}_{\eps 1}^k+\sum_{k=0}^3\tilde{\Gamma}^0\tilde{\Gamma}^k\K_{\eps 2}^kF_1\otimes G_0-F_1\otimes G_0\sum_{k=0}^3\tilde{\Gamma}^0\tilde{\Gamma}^k\tilde{\K}_{\eps 2}^k\}.
\end{aligned}
\end{equation*}
With $\mu_\eps=\mu_0+\sqrt{\eps}\mu_{1,\eps}+\eps\mu_{2,\eps}$, we have
\begin{equation*}
\G_{\mu_0}^{2,\eps}[W](t)=\G_{\mu_0}^2[W](t)+\langle \bar{C}_1,W\otimes W\rangle(t)+\int_0^t ds\langle \sum_{i=1}^4 \bar{C}_i, W\otimes W\rangle(s),
\end{equation*}
and the correctors $\bar{C}_i$ are
\begin{eqnarray*}
&&\bar{C}_1=\sqrt{\eps}\mu_{1,\eps}+\eps\mu_{2,\eps},\\
&&\bar{C}_2=-\partial_t(\sqrt{\eps}\mu_{1,\eps}+\eps\mu_{2,\eps}),\\
&&\bar{C}_3=ie\sqrt{\eps}\sum_{k=0}^3(\Gamma^0\Gamma^k\K_{\eps 1}^k+\tilde{\Gamma}^0\tilde{\Gamma}^k\K_{\eps 2}^k)\mu_{2,\eps}-ie\sqrt{\eps}\mu_{2,\eps}\sum_{k=0}^3(\Gamma^0\Gamma^k\tilde{\K}_{\eps 1}^k+\tilde{\Gamma}^0\tilde{\Gamma}^k\tilde{\K}_{\eps 2}^k),\\
&&\bar{C}_4=-\mathbf{L}_2\mu_0,
\end{eqnarray*}
\begin{equation*}
\begin{aligned}
\bar{C}_5=&\frac{1}{\eps}\sum_{k=1}^3c\left[\Gamma^0\Gamma^kP_{k1}^*(\frac{\eps D_x}{2})+\tilde{\Gamma}^0\tilde{\Gamma}^kP_{k2}^*(\frac{\eps D_x}{2})\right](\sqrt{\eps}\mu_{1,\eps}+\eps\mu_{2,\eps})\\
+&\frac{1}{\eps}(\sqrt{\eps}
\mu_{1,\eps}+\eps\mu_{2,\eps})\sum_{k=1}^3c\left[
\Gamma^0\Gamma^kP_{k1}(-\frac{\eps D_x}{2})+\tilde{\Gamma}^0\tilde{\Gamma}^kP_{k2}(-\frac{\eps D_x}{2})\right].
\end{aligned}
\end{equation*}
\begin{remark}
In $\bar{C}_5$, $D_x$ is with respect to the slow variable.
\end{remark}

To prove that $\G_{\mu_0}^2[W](t)$ is a $\tilde{\Pb}_\eps$-approximating martingale, we need to show that
\begin{equation*}
\|\langle \sum_{i=1}^5\bar{C}_i,\sum_{i=1}^5 \bar{C}_i\rangle(t)\|_{L^\infty(\V)} \leq C_{\mu_0,T}\eps
\end{equation*}
uniformly in $t\in [0,T]$.

\section{Proof of the main theorem}
\label{sec:proof}
In this section, we first prove the convergence of the cross modes and then of the propagating modes. In the end, we prove the tightness result and finish the proof of the main theorem.

From now on, we use the notation $a\les b$ when there exists a constant $M$ such that $a\leq Mb$.

\subsection{Cross modes}
First of all, we show that the cross modes converge to zero weakly almost everywhere. For any test function $F$, we have
\begin{equation*}
\partial_t\langle F,W_\eps\rangle=\langle (\partial_t+\A_\eps)F,W_\eps\rangle.
\end{equation*}
Therefore,
\begin{equation*}
\eps\langle F,W_\eps\rangle(t)-\eps\langle F,W_\eps\rangle(0)=\int_0^t ds\langle (\eps\partial_t+\eps\A_\eps)F,W_\eps\rangle(s).
\end{equation*}
By \eqref{eq:adjointAeps}, we rewrite $\eps\A_\eps F$ and have
\begin{equation*}
\begin{aligned}
\eps \A_\eps F=&(\sum_{k=1}^3c\gamma^0\gamma^ki\xi_k+im_0c^2\gamma^0)F-F(\sum_{k=1}^3i\xi_kc\gamma^0\gamma^k+im_0c^2\gamma^0)\\
-& \eps\sum_{k=1}^3c\gamma^0\gamma^kP_k^*(\frac{D}{2})F-\eps F\sum_{k=1}^3P_k(-\frac{D}{2})c\gamma^0\gamma^k-\sqrt{\eps}ie\sum_{k=0}^3\gamma^0\gamma^k\K_{\eps}^kF+\sqrt{\eps}ie\sum_{k=0}^3\tilde{\K}_{\eps}^kF\gamma^0\gamma^k.
\end{aligned}
\end{equation*}
 Since $\langle W_\eps(t),W_\eps(t)\rangle \leq M$ uniformly in $t\in [0,T]$ and $\tilde{A}\in \V$, $\K_\eps^k,\tilde{\K}_\eps^k$ are bounded operators on $L^2(\R^{2d})$, and $Q(\xi)=\sum_{k=1}^3 \gamma^0\gamma^k\xi_k+m_0c\gamma^0$, we have
\begin{equation*}
|\int_0^T ds \langle F, QW_\eps-W_\eps Q\rangle(s) |=|\int_0^Tds\langle QF-FQ, W_\eps\rangle(s)|\les \sqrt{\eps}.
\end{equation*}
Recall that
\begin{equation*}
W_\eps=\sum_{i,j=1,2} a_{ij}^\eps x_ix_j^*+\sum_{i,j=1,2} b_{ij}^\eps y_iy_j^*+\sum_{i,j=1,2} c_{ij}^\eps x_iy_j^*+\sum_{i,j=1,2} d_{ij}^\eps y_ix_j^*,
\end{equation*}
 so $QW_\eps-W_\eps Q=2\lambda_+\sum_{i,j=1,2} c_{ij}^\eps x_iy_j^*+2\lambda_-\sum_{i,j=1,2} d_{ij}^\eps y_ix_j^*$. If we choose $F=fx_iy_j^*/\lambda_+$ for some good function $f$, we have
 \begin{equation*}
 |\int_0^T ds\langle f, c_{ij}^\eps\rangle(s) |\les \sqrt{\eps}.
 \end{equation*}
 The same is true for $d_{ij}^\eps$ if we choose other test functions. Thus we conclude that the cross modes $c_{ij}^\eps,d_{ij}^\eps$ converge to zero weakly almost everywhere.
 \begin{remark}
 We do not necessarily have the weak convergence of $c_{ij}^\eps, d_{ij}^\eps$ as processes in $\C([0,T];L^2(\R^{2d}))$ here. The following is a heuristic argument.

 We rewrite the equation \eqref{eq:WignerEq} satisfied by $W_\eps$ :
 \begin{equation}
 \begin{aligned}
 &\partial_t W_\eps +\frac{1}{\eps}\left[\sum_{k=1}^3 c\gamma^0\gamma^kP_k(i\xi)W_\eps
 +\sum_{k=1}^3 W_\eps P_k^*(i\xi)c\gamma^0\gamma^k+im_0c^2(\gamma^0W_\eps-W_\eps\gamma^0)\right]
 +\ldots\\
  =&\partial_t W_\eps+\frac{ic}{\eps}(QW_\eps-W_\eps Q)+\ldots=0.
\label{eq:Eqwith1overEps}
  \end{aligned}
 \end{equation}
 All the terms that do not show up here are those of order $1$ or $1/\sqrt{\eps}$. So we have
 \begin{equation*}
 x_i^*\partial_t W_\eps y_j+\frac{ic}{\eps} x_i^*(Q W_\eps-W_\eps Q)y_j+\ldots =0,
 \end{equation*}
 which leads to
 \begin{equation*}
 \partial_t c_{ij}^\eps+\frac{2ic\lambda_+}{\eps}c_{ij}^\eps+\ldots =0.
 \end{equation*}
 We see that as $\eps \to 0$, the cross modes $c_{ij}^\eps$ are highly oscillatory in time. The same happens to $d_{ij}^\eps$. From this perspective, we can only expect the weak convergence as a $L^2$ function rather than a process in $\C([0,T];L^2(\R^{2d}))$.
 \end{remark}

\subsection{Convergence of the expectation}
We first define some notation used in the proof. For two matrix-valued functions $A(y)$ and $B(y)$, we have
\begin{eqnarray*}
&&|A|\leq |B|\Leftrightarrow \forall (m,n), |A_{mn}(y)|\leq |B_{mn}(y)|\\
&&|A|\leq \max_{|y|<M} |B| \Leftrightarrow \forall (m,n), |A_{mn}(y)|\leq \max_{|y|<M}|B_{mn}(y)|\\
&&|A|\leq \|B\|_{L^\infty(\V)}\Leftrightarrow \forall (m,n), |A_{mn}(y)|\leq \|B_{mn}(y)\|_{L^\infty(\V)}
\end{eqnarray*}

We have the following lemmas concerning $\lambda_{1,\eps},\lambda_{2,\eps}$.
\begin{lemma}
$\langle \lambda_{1,\eps}(t),\lambda_{1,\eps}(t)\rangle \leq C_{\lambda_0,T}$ uniformly for $t\in [0,T]$ and $\tilde{A}\in \V$.
\label{lem:lambda1}
\end{lemma}
\begin{proof}
Recall from \eqref{eq:lambda1}, we have $\lambda_{1,\eps}(t,x,\xi)=\lambda_1(t,x,x/\eps,\xi)$, and
\begin{equation*}
\begin{aligned}
&\lambda_1(t,x,z,\xi)\\
=&-\int_0^\infty dre^{r\Q}\int_{\R^d}\frac{dpe^{iz\cdot p}}{(2\pi)^d}e^{-rA_1(\xi,p)}ie\left[G_t(p)\lambda_0(\xi-\frac{p}{2})-\lambda_0(\xi+\frac{p}{2})G_t(p)\right]e^{-rA_2(\xi,p)}.
\end{aligned}
\end{equation*}

Take one term, for example
\begin{equation*}
(I)=-\int_0^\infty dr e^{r\Q}\int_{\R^d}\frac{dp\tilde{A}_k(\frac{t}{\eps},p)}{(2\pi)^d}
e^{iz\cdot p}iee^{-rA_1(\xi,p)}\gamma^0\gamma^k\lambda_0(\xi-\frac{p}{2})e^{-rA_2(\xi,p)}.
\end{equation*}
Since $\tilde{A}_k(\frac{t}{\eps},p)dp$ is compactly supported and of bounded total variation, by \eqref{eq:speGap} we have
\begin{equation*}
|(I)|\les
\int_0^\infty dr e^{-\alpha r} \max_{|p|<M} |e^{-rA_1(\xi,p)}\gamma^0\gamma^k\lambda_0(\xi-\frac{p}{2})e^{-rA_2(\xi,p)}|.
\end{equation*}

Because $A_1(\xi,p)=-icQ(\xi+\frac{p}{2}),A_2(\xi,p)=icQ(\xi-\frac{p}{2})$, any element in $e^{-rA_1(\xi,p)}$ and $e^{-rA_2(\xi,p)}$ is uniformly bounded and thus
 \begin{equation*}
    |\left[e^{-rA_1(\xi,p)}\gamma^0\gamma^k\lambda_0(\xi-\frac{p}{2})e^{-rA_2(\xi,p)}\right]_{mn} | \les\sum_{m,n}|\left[\lambda_0(\xi-\frac{p}{2})\right]_{mn}|.
    \end{equation*}
     Since $\lambda_0$ is a good function, the proof is completed.
\end{proof}

\begin{lemma}
$\langle \lambda_{2,\eps}(t),\lambda_{2,\eps}(t)\rangle \leq C_{\lambda_0,T}$ for $t\in [0,T]$ and $\tilde{A}\in \V$.
\label{lem:lambda2}
\end{lemma}
\begin{proof}
The proof is similar with the one for $\lambda_{1,\eps}$. Recall that \begin{equation*}
\Q\hat{\lambda}_2-A_1\hat{\lambda}_2-\hat{\lambda}_2A_2=F_2-\E\{F_2\},
\end{equation*} where
\begin{equation*}
F_2=\frac{ie}{(2\pi)^d}\int_{\R^d}G_t(q)\hat{\lambda}_1(p-q,\xi-\frac{q}{2})dq
-\frac{ie}{(2\pi)^d}\int_{\R^d}\hat{\lambda}_1(p-q,\xi+\frac{q}{2})G_t(q)dq,
\end{equation*}
and \begin{equation*}
\hat{\lambda}_1(p,\xi)=-\int_0^\infty dr e^{r\Q}e^{-rA_1(\xi,p)}\left[ieG_t(p)\lambda_0(\xi-\frac{p}{2})-ie\lambda_0(\xi+\frac{p}{2})G_t(p)\right]e^{-rA_2(\xi,p)},
\end{equation*}
therefore we can write $F_2=\frac{-e^2}{(2\pi)^d}(B_1-B_2)$, where
\begin{equation*}
B_1=-\int_{\R^d}dqG_t(q)\int_0^\infty dr e^{r\Q}e^{-rA_1(\xi-\frac{q}{2},p-q)}\left[
G_t(p-q)\lambda_0(\xi-\frac{p}{2})-\lambda_0(\xi+\frac{p}{2}-q)G_t(p-q)\right]e^{-rA_2(\xi-\frac{q}{2},p-q)},
\end{equation*}
and
\begin{equation*}
B_2=-\int_{\R^d}dq\int_0^\infty dr e^{r\Q}e^{-rA_1(\xi+\frac{q}{2},p-q)}\left[
G_t(p-q)\lambda_0(\xi+q-\frac{p}{2})-\lambda_0(\xi+\frac{p}{2})G_t(p-q)\right]
e^{-rA_2(\xi+\frac{q}{2},p-q)}G_t(q).
\end{equation*}
Then we can write
\begin{equation*}
\lambda_2(z,\xi)=-\int_0^\infty dr e^{r\Q}\int_{\R^d}\frac{dpe^{iz\cdot p}}{(2\pi)^d}e^{-rA_1}\frac{-e^2}{(2\pi)^d}(B_1-B_2-\E B_1+\E B_2)e^{-rA_2}.
\end{equation*}
Considering the term $B_1-\E B_1$, we have
\begin{equation*}
\begin{aligned}
&|\int_0^\infty dr e^{r\Q}\int_{\R^d}dpe^{iz\cdot p}e^{-rA_1}(B_1-\E B_1)e^{-rA_2}|\leq \int_0^\infty dr e^{-\alpha r}\|\int_{\R^d}dpe^{iz\cdot p}e^{-rA_1}(B_1-\E B_1)e^{-rA_2}\|_{L^\infty(\V)}\\
\leq & \int_0^\infty dr e^{-\alpha r}\|\int_{\R^d}dpe^{iz\cdot p}e^{-rA_1}B_1e^{-rA_2}\|_{L^\infty(\V)}+\int_0^\infty dr e^{-\alpha r}\|\int_{\R^d}dpe^{iz\cdot p}e^{-rA_1}\E B_1e^{-rA_2}\|_{L^\infty(\V)}.
\end{aligned}
\end{equation*}
Taking $\|\int_{\R^d} dpe^{iz\cdot p}e^{-rA_1}B_1e^{-rA_2}\|_{L^\infty(\V)}$ for example, we have
\begin{equation*}
\begin{aligned}
 &\int_{\R^d}dpe^{iz\cdot p}e^{-rA_1(\xi,p)}B_1e^{-rA_2(\xi,p)}\\
 =-\int_{\R^d}&dpe^{iz\cdot p}e^{-rA_1(\xi,p)}\int_{\R^d}dqG_t(q)\int_0^\infty dr e^{r\Q}e^{-rA_1(\xi-\frac{q}{2},p-q)}\\
 &\left[G_t(p-q)\lambda_0(\xi-\frac{p}{2})-\lambda_0(\xi+\frac{p}{2}-q)G_t(p-q)\right]e^{-rA_2(\xi-\frac{q}{2},p-q)}e^{-rA_2(\xi,p)}.
 \end{aligned}
 \end{equation*}
 Since $G_t(q)=\sum_{k=0}^3\gamma^0\gamma^k\tilde{A}_k(\frac{t}{\eps},q)$, we have
 \begin{equation*}
 \begin{aligned}
&|\int_{\R^d}dq\tilde{A}_{k_1}(\frac{t}{\eps},q)\int_0^\infty dr e^{r\Q}\\
&\int_{\R^d}dpe^{iz\cdot p}e^{-rA_1(\xi,p)}\gamma^0\gamma^{k_1}e^{-rA_1(\xi-\frac{q}{2},p-q)}\gamma^0\gamma^{k_2}
\tilde{A}_{k_2}(\frac{t}{\eps},p-q)\lambda_0(\xi-\frac{p}{2})e^{-rA_2(\xi-\frac{q}{2},p-q)}e^{-rA_2(\xi,p)}|\\
\les&\int_{\R^d}dq|\tilde{A}_{k_1}(\frac{t}{\eps},q)|\int_0^\infty dr e^{-\alpha r}\max_{|p-q|<M}|e^{-rA_1(\xi,p)}\gamma^0\gamma^{k_1}e^{-rA_1(\xi-\frac{q}{2},p-q)}\gamma^0\gamma^{k_2}
\lambda_0(\xi-\frac{p}{2})e^{-rA_2(\xi-\frac{q}{2},p-q)}e^{-rA_2(\xi,p)}|,
\end{aligned}
 \end{equation*}
and so similarly we have
 \begin{equation*}
 \begin{aligned}
 &|\left[e^{-rA_1(\xi,p)}\gamma^0\gamma^{k_1}e^{-rA_1(\xi-\frac{q}{2},p-q)}\gamma^0\gamma^{k_2}
\lambda_0(\xi-\frac{p}{2})e^{-rA_2(\xi-\frac{q}{2},p-q)}e^{-rA_2(\xi,p)}\right]_{mn}|\\
\les &
\sum_{m,n}|\left[\lambda_0(\xi-\frac{p}{2})\right]_{mn}|,
\end{aligned}
\end{equation*}
 which completes the proof.
\end{proof}

By Lemma \ref{lem:lambda1} and \ref{lem:lambda2}, we have $\|\langle C_1,C_1\rangle(t)\|_{L^\infty(\V)} \leq C_{\lambda_0,T}\eps$. For $C_i,i=2,3,4$, the proof is similar so we omit the details here. In summary, we have
\begin{equation*}
\|\langle \sum_{i=1}^4 C_i,\sum_{i=1}^4 C_i\rangle(t)\|_{L^\infty(\V)} \leq C_{\lambda_0,T}\eps
\end{equation*}
uniformly in $t\in [0,T]$, i.e., the $L^2$ norm of $C_i$ is of order $\sqrt{\eps}$, so $\G_{\lambda_0}^1[W](t)$ is an approximating martingale under $\tilde{\Pb}_\eps$, i.e.,
 \begin{equation*}
 |\E^{\tilde{\Pb}_\eps}\{\G^1_{\lambda_0}[W](t)|\F_s\}-\G^1_{\lambda_0}[W](s)|\leq C_{\lambda_0,T}\sqrt{\eps}.
 \end{equation*}

 Now we can prove the following proposition about convergence of expectation.
 \begin{proposition}
 If $(\alpha_+^\eps,\alpha_-^\eps)$ converges weakly to some $(\alpha_+,\alpha_-)$ in $\C([0,T];H_M)$, then $(\E\alpha_+,\E\alpha_-)$ is the weak solution to \eqref{eq:diracEqnew}.
 \label{prop:1stMomentCon}
 \end{proposition}
\begin{proof}
First of all, we point out that $(\E\alpha_+,\E\alpha_-)$ has a fixed initial condition by our choice of $W_\eps(0,x,\xi)$. By choosing the test function to be $\lambda_0=f_0(x_1x_1^*+x_2x_2^*)$, we have
\begin{equation*}
\langle \lambda_0, \E W_\eps\rangle(0)=\langle f_0,\E \alpha_+^\eps\rangle(0).
\end{equation*}
Since $\langle \lambda_0, \E W_\eps\rangle(0)\to \langle\lambda_0,  W_0\rangle(0)=\langle f_0, \Tr(\Pi_+ W_0)\rangle(0)$ and $ \langle f_0,\E \alpha_+^\eps\rangle(0)\to \langle f_0, \E\alpha_+\rangle(0)$, we obtain $\E\alpha_+(0)=\Tr(\Pi_+ W_0(0))$. The same discussion holds for $\E\alpha_-(0)$.

In the first approximating martingale inequality, we choose $s=0$,
\begin{equation*}
|\langle \lambda_0,\E W_\eps\rangle(t)-\langle \lambda_0,\E W_\eps\rangle(0)-\int_0^tds \langle (\partial_t+\A)\lambda_0, \E W_\eps\rangle(s)|\leq C_{\lambda_0,T}\sqrt{\eps}.
\end{equation*}
By the choice of $\lambda_0$, we further obtain
\begin{equation*}
\begin{aligned}
&\langle \lambda_0,\E W_\eps\rangle(t)-\langle \lambda_0,\E W_\eps\rangle(0)-\int_0^tds\langle \partial_t \lambda_0, \E W_\eps\rangle(s)\\
=&
\langle f_0, \E \alpha_+^\eps\rangle(t)-\langle f_0,\E \alpha_+^\eps\rangle(0)-\int_0^t ds\langle \partial_t f_0, \E\alpha_+^\eps\rangle(s)\\
\to& \langle f_0,\E\alpha_+\rangle(t)-\langle f_0, \E \alpha_+\rangle(0)-\int_0^t ds\langle \partial_t f_0, \E \alpha_+\rangle(s)
\end{aligned}
\end{equation*}
as $\eps \to 0$.

For the term $\int_0^t ds \langle \A\lambda_0,\E W_\eps\rangle(s)$, we need to calculate $x_1^*\A^*(\E W_\eps)x_1+x_2^*\A^*(\E W_\eps)x_2$. After some lengthy algebra, we get
\begin{equation*}
x_1^*\A^*(\E W_\eps)x_1+x_2^*\A^*(\E W_\eps)x_2
=-\frac{c\xi\cdot\nabla_x\E\alpha_+^\eps}{\lambda_+(\xi)}+\mathcal{T}(\E\alpha_+^\eps,\E\alpha_-^\eps)+(I),
\end{equation*}
where $(I)$ includes terms containing $c_{ij}^\eps,d_{ij}^\eps$. By the bound of cross modes, we check that $\int_0^t ds\langle f_0, (I)\rangle(s)\to 0$ as $\eps \to 0$. Thus, we have shown that $(\E\alpha_+,\E\alpha_-)$ is the weak solution to
\begin{equation*}
\partial_t\E\alpha_++\frac{c\xi\cdot \nabla_x\E\alpha_+}{\lambda_+(\xi)}=\mathcal{T}(\E\alpha_+,\E\alpha_-)
\end{equation*}
with the initial condition given by $\E\alpha_\pm(0)=\Tr(\Pi_\pm W_0(0))$.

In the same way, if we choose $\lambda_0=f_0(y_1y_1^*+y_2y_2^*)$, we can show $(\E\alpha_+,\E\alpha_-)$ also satisfies
\begin{equation*}
\partial_t\E\alpha_-+\frac{c\xi\cdot \nabla_x\E\alpha_-}{\lambda_-(\xi)}=\mathcal{T}(\E\alpha_-,\E\alpha_+).
\end{equation*}
By the uniqueness of the solution to the above transport equations, the proof is complete.
\end{proof}

\subsection{Convergence of the second moment}
We now prove the second approximating martingale inequality. Recalling the construction of $\mu_\eps=\mu_0+\sqrt{\eps}\mu_{1,\eps}+\eps\mu_{2,\eps}$, we need to prove that the correctors $\bar{C}_i,i=1,\ldots,5$ are small.

For $\bar{C}_4$, we have the following lemma.
\begin{lemma}
$\langle \bar{C}_4,\bar{C}_4\rangle(t)\leq C_{\mu_0,T}\eps^d$ uniformly in $t\in [0,T]$ and $\tilde{A}\in \V$.
\end{lemma}
\begin{proof}
Recall that
\begin{equation*}
\mathbf{L}_2 \mu_0=ie\E\{-\sum_{k=0}^3\Gamma^0\Gamma^k\K_{\eps 1}^kF_0\otimes G_1+F_0\otimes G_1
\sum_{k=0}^3\Gamma^0\Gamma^k\tilde{\K}_{\eps 1}^k-\sum_{k=0}^3\tilde{\Gamma}^0\tilde{\Gamma}^k\K_{\eps 2}^kF_1\otimes G_0+F_1\otimes G_0\sum_{k=0}^3\tilde{\Gamma}^0\tilde{\Gamma}^k\tilde{\K}_{\eps 2}^k\}.
\end{equation*}
Consider one term, for example
\begin{equation*}
\sum_{k=0}^3\Gamma^0\Gamma^k \K_{\eps 1}^kF_0\otimes G_1=\sum_{k=0}^3\left(\gamma^0\gamma^k\int_{\R^d}\frac{dp\tilde{A}_k(\frac{t}{\eps},p)}{(2\pi)^d}e^{ix\cdot p/\eps}
F_0(x,\xi-\frac{p}{2})\right)\otimes G_1.
\end{equation*}
Since \begin{equation*}
\begin{aligned}
&G_1(t,x,\frac{x}{\eps},\xi)\\
=&-\int_0^\infty dre^{r\Q}\int_{\R^d}\frac{dpe^{ix\cdot p/\eps}}{(2\pi)^d}e^{-rA_1(\xi,p)}ie\left[G_t(p)G_0(\xi-\frac{p}{2})-G_0(\xi+\frac{p}{2})G_t(p)\right]e^{-rA_2(\xi,p)},
\end{aligned}
\end{equation*}
and $G_t(p)=\sum_{k=0}^3\gamma^0\gamma^k\tilde{A}_k(\frac{t}{\eps},p)$, taking one term in the sum and ignoring constant, we have
\begin{equation*}
\begin{aligned}
(I)=&\left(\gamma^0\gamma^m
\int_{\R^d}\frac{dp\tilde{A}_m(\frac{t}{\eps},p)}{(2\pi)^d}e^{ix\cdot p/\eps}
F_0(x,\xi-\frac{p}{2})\right)\otimes\\
&\left(\int_0^\infty dr e^{r\Q}\int_{\R^d}\frac{dp\tilde{A}_n(\frac{t}{\eps},p)}{(2\pi)^d}e^{ix\cdot p/\eps}e^{-rA_1(\xi,p)}\gamma^0\gamma^nG_0(x,\xi-\frac{p}{2})e^{-rA_2(\xi,p)}\right).
\end{aligned}
\end{equation*}
For any element in matrix $\E(I)$, we know that it is a linear combination of terms of the following form
\begin{equation*}
(i)=\int_0^\infty dr \int_{\R^d}dp e^{\frac{i(x_1-x_2)\cdot p}{\eps}}\tilde{R}_{mn}(r,p)f(x_1,\xi_1-\frac{p}{2})
g(x_2,\xi_2+\frac{p}{2})T_1(r,\xi_2,p)T_2(r,\xi_2,p),
\end{equation*}
and $f,g,T_1,T_2$ are from $F_0,G_0,e^{-rA_1(\xi_2,p)},e^{-rA_2(\xi_2,-p)}$ respectively and we can assume they are all real. So
\begin{equation*}
\|(i)\|_{L^2(\R^{2d})}^2
=\int_{\R^{6d}}\int_0^\infty\int_0^\infty
H\tilde{R}_{mn}(s_1,p_1)\tilde{R}_{mn}(s_2,p_2)e^{\frac{i(x_1-x_2)\cdot (p_1-p_2)}{\eps}}ds_1ds_2dp_1dp_2dx_1dx_2d\xi_1d\xi_2,
\end{equation*}
where
\begin{equation*}
H=\prod_{i=1,2}f(x_1,\xi_1-\frac{p_i}{2})\prod_{i=1,2}g(x_2,\xi_2+\frac{p_i}{2})
\prod_{i=1,2}T_i(s_1,\xi_2,p_1)\prod_{i=1,2}T_i(s_2,\xi_2,p_2).
\end{equation*}
By density argument, we can assume $f(x_1,\xi_1)g(x_2,\xi_2)=h_1(x_1+x_2)h_2(x_1-x_2)h_3(\xi_1)h_4(\xi_2)$ for some good function $h_i$, then we have
\begin{equation*}
H=|h_1(x_1+x_2)|^2|h_2(x_1-x_2)|^2\prod_{i=1,2}h_3(\xi_1-\frac{p_i}{2})
\prod_{i=1,2}h_4(\xi_2+\frac{p_i}{2})\prod_{i=1,2}T_i(s_1,\xi_2,p_1)\prod_{i=1,2}T_i(s_2,\xi_2,p_2).
\end{equation*}
Change variables $y_1=x_1+x_2,y_2=x_1-x_2$, and integrate in $y_i$, we have
\begin{equation*}
\begin{aligned}
\|(i)\|_{L^2(\R^{4d})}^2
=C\int_{\R^{4d}}\int_0^\infty\int_0^\infty&\prod_{i=1,2}h_3(\xi_1-\frac{p_i}{2})
\prod_{i=1,2}h_4(\xi_2+\frac{p_i}{2})\prod_{i=1,2}T_i(s_1,\xi_2,p_1)\prod_{i=1,2}T_i(s_2,\xi_2,p_2)\\
&\tilde{R}_{mn}(s_1,p_1)\tilde{R}_{mn}(s_2,p_2)\hat{\nu}(\frac{p_2-p_1}{\eps})
ds_1ds_2dp_1dp_2d\xi_1d\xi_2
\end{aligned}
\end{equation*}
for some constant $C$ and $\hat{\nu}$ is the Fourier transform of of $|h_2|^2$. Since $T_i$ is bounded, integrating in $\xi_1$ and $\xi_2$ yields:
\begin{equation*}
\begin{aligned}
\|(i)\|_{L^2(\R^{4d})}^2
\les\int_{\R^{2d}}\int_0^\infty\int_0^\infty
|\tilde{R}_{mn}(s_1,p_1)\tilde{R}_{mn}(s_2,p_2)\hat{\nu}(\frac{p_2-p_1}{\eps})|
ds_1ds_2dp_1dp_2.
\end{aligned}
\end{equation*}

Changing variable $p_2=p_1+u$, then integrating in $p_1,s_1,s_2$, since $\tilde{R}_{mn}\in \mathcal{S}(\R\times\R^d)$, we have
\begin{equation*}
\|(i)\|_{L^2(\R^{4d})}^2
\les \int_{\R^d}|\hat{\nu}(\frac{u}{\eps})|du.
\end{equation*}
Therefore we see that
\begin{equation*}
\|(i)\|_{L^2(\R^{4d})}^2
\les \eps^d.
\end{equation*}
\end{proof}

Since $\mu_1=F_0\otimes G_1+F_1\otimes G_0$ and $\mu_2=F_0\otimes G_2+F_2\otimes G_0+\tilde{\mu}_2$, and we already have the control for $F_1,G_1,F_2,G_2$, we only have to show \begin{equation*}
\langle \tilde{\mu}_{2,\eps},\tilde{\mu}_{2,\eps}\rangle \leq C_{\mu_0,T}
\end{equation*}
to conclude that $\langle \bar{C}_i, \bar{C}_i\rangle \leq C_{\mu_0,T} \eps$ for $i=1,2,3,5$. The equation satisfied by $\tilde{\mu}_2$ is
\begin{equation*}
\begin{aligned}
 &\Q\tilde{\mu}_2-\sum_{k=1}^3c\left[\Gamma^0\Gamma^k P_{k1}^*(i\xi+\frac{D_z}{2})+\tilde{\Gamma}^0\tilde{\Gamma}^kP_{k2}^*(i\xi+\frac{D_z}{2})\right]\tilde{\mu}_2+im_0c^2(\Gamma^0+\tilde{\Gamma}^0)\tilde{\mu}_2\\
-&\tilde{\mu}_2\sum_{k=1}^3c\left[\Gamma^0\Gamma^kP_{k1}(i\xi-\frac{D_z}{2})+\tilde{\Gamma}^0\tilde{\Gamma}^kP_{k2}(i\xi-\frac{D_z}{2})\right]-im_0c^2\tilde{\mu}_2(\Gamma^0+\tilde{\Gamma}^0)\\
=& \mathbf{L}_2\mu_0+ie\E\{\sum_{k=0}^3\Gamma^0\Gamma^k\K_{\eps 1}^kF_0\otimes G_1-F_0\otimes G_1
\sum_{k=0}^3\Gamma^0\Gamma^k\tilde{\K}_{\eps 1}^k+\sum_{k=0}^3\tilde{\Gamma}^0\tilde{\Gamma}^k\K_{\eps 2}^kF_1\otimes G_0-F_1\otimes G_0\sum_{k=0}^3\tilde{\Gamma}^0\tilde{\Gamma}^k\tilde{\K}_{\eps 2}^k\}.
\end{aligned}
\end{equation*}
We just have to note that the matrices
\begin{eqnarray*}
\mathbf{A}_1:=\sum_{k=1}^3c\left[\Gamma^0\Gamma^k P_{k1}^*(i\xi+\frac{ip}{2})+\tilde{\Gamma}^0\tilde{\Gamma}^kP_{k2}^*(i\xi+\frac{ip}{2})\right]
-im_0c^2(\Gamma^0+\tilde{\Gamma}^0)\\
\mathbf{A}_2:=\sum_{k=1}^3c\left[\Gamma^0\Gamma^kP_{k1}(i\xi-\frac{ip}{2})+\tilde{\Gamma}^0\tilde{\Gamma}^kP_{k2}(i\xi-\frac{ip}{2})\right]
+im_0c^2(\Gamma^0+\tilde{\Gamma}^0)
\end{eqnarray*}
are both of the form $iQ$ for some real symmetric matrix $Q$, thus any element in the matrices $e^{-r\mathbf{A}_1}, e^{-r\mathbf{A}_2}$ is bounded. The rest of the proof is similar to the one for $\lambda_{2,\eps}$.

With the second approximating martingale inequality, we can prove the following proposition.
\begin{proposition}
If $(\alpha_+^\eps,\alpha_-^\eps)$ converges weakly to $(\alpha_+,\alpha_-)$ in $\C([0,T];H_M)$, then $(\alpha_+,\alpha_-)$ is the unique weak solution to \eqref{eq:diracEqnew}.
\label{prop:2ndMomentCon}
\end{proposition}
\begin{proof}
First of all, we claim that $\alpha_\pm(0,x,\xi)=\Tr(\Pi_\pm W_0(0,x,\xi))$. By the proof in Proposition \ref{prop:1stMomentCon}, $\E\alpha_\pm(0,x,\xi)=\Tr(\Pi_\pm W_0(0,x,\xi))$, so we only have to show that $\alpha_\pm(0)$ is deterministic, and this comes from the fact that
\begin{equation*}
\langle f_0\otimes g_0, \E\{\alpha_+\}\otimes \E\{\alpha_+\}\rangle(0)=\langle F_0\otimes G_0, W_0\otimes W_0\rangle(0)= \langle f_0\otimes g_0, \E\{\alpha_+\otimes \alpha_+\}\rangle(0),
\end{equation*}
since they are all the limit of $\langle F_0\otimes G_0, \E\{W_\eps\otimes W_\eps\}\rangle(0)$ and we have chosen $F_0=f_0(x_1x_1^*+x_2x_2^*), G_0=g_0(x_1x_1^*+x_2x_2^*)$. The same discussion holds for $\alpha_-$ if we choose $F_0,G_0$ with $x_i$ replaced by $y_i$.

In the second approximating martingale inequality \eqref{eq:2ndapproMartin}, let $s=0$,
\begin{equation*}
\begin{aligned}
|&\langle F_0\otimes G_0,\E\{W_\eps\otimes W_\eps\}\rangle(t)
-\langle F_0\otimes G_0,\E\{W_\eps\otimes W_\eps\}\rangle(0)\\
&-\int_0^t ds
\langle F_0\otimes \left[(\partial_t +\A)G_0\right]+\left[(\partial_t+\A)F_0\right]\otimes G_0, \E\{ W_\eps \otimes W_\eps\}\rangle (s)|\les \sqrt{\eps}.
\end{aligned}
\end{equation*}
and by the aforementioned $F_0,G_0$, we have
\begin{equation*}
\langle F_0\otimes G_0,\E\{W_\eps\otimes W_\eps\}\rangle=\langle f_0\otimes g_0, \E\{\alpha_+^\eps\otimes\alpha_+^\eps\}\rangle
\end{equation*}
and
\begin{equation*}
\begin{aligned}
&\langle F_0\otimes \left[(\partial_t +\A)G_0\right]+\left[(\partial_t+\A)F_0\right]\otimes G_0, \E\{ W_\eps \otimes W_\eps\}\rangle\\
=&\langle f_0\otimes g_0, \E\{\alpha_+^\eps\otimes (x_1^*\A^*W_\eps x_1+x_2^*\A^* W_\eps x_2)+
(x_1^*\A^*W_\eps x_1+x_2^*\A^* W_\eps x_2)\otimes \alpha_+^\eps\}\rangle\\
+&\langle f_0\otimes \partial_t g_0+\partial_t f_0\otimes g_0, \E\{\alpha_+^\eps\otimes \alpha_+^\eps\}\rangle
\end{aligned}
\end{equation*}
By the same discussion as in Proposition \ref{prop:1stMomentCon}, let $\eps\to 0$, we have
\begin{equation*}
\begin{aligned}
&\langle f_0\otimes g_0, \E\{\alpha_+\otimes \alpha_+\}\rangle(t)
-\langle f_0\otimes g_0, \E\{\alpha_+\otimes \alpha_+\}\rangle(0)\\
=&\int_0^t ds \langle (\partial_t+\frac{c\xi\cdot \nabla_x}{\lambda_+(\xi)})f_0\otimes g_0+f_0\otimes (\partial_t+\frac{c\xi\cdot \nabla_x}{\lambda_+(\xi)})g_0+,\E\{\alpha_+\otimes\alpha_+\}\rangle(s)\\
+&\int_0^t ds\langle f_0\otimes g_0, \E\{\alpha_+\otimes \mathcal{T}(\alpha_+,\alpha_-)+
\mathcal{T}(\alpha_+,\alpha_-)\otimes \alpha_+\}\rangle(s).
\end{aligned}
\end{equation*}
Note that we can define some operator such that $\alpha_+\otimes \mathcal{T}(\alpha_+,\alpha_-)+
\mathcal{T}(\alpha_+,\alpha_-)\otimes \alpha_+$ could be written as a functional of $(\alpha_+\otimes \alpha_+, \alpha_+\otimes \alpha_-,\alpha_-\otimes \alpha_+)$. Therefore, we have derived an equation satisfied by $(\E\{\alpha_+\otimes \alpha_+\}, \E\{\alpha_+\otimes \alpha_-\}, \E\{\alpha_-\otimes \alpha_+\})$. By the result from Proposition \ref{prop:1stMomentCon}, we check that $(\E\alpha_+\otimes \E\alpha_+, E\alpha_+\otimes \E\alpha_-, \E\alpha_-\otimes \E\alpha_+)$ satisfies the same equation.

By choosing other forms of $F_0,G_0$, we can derive an equation system satisfied by $(\E\{\alpha_+\otimes \alpha_+\},\E\{\alpha_+\otimes \alpha_-\},\E\{\alpha_-\otimes \alpha_+\},\E\{\alpha_-\otimes \alpha_-\})$. We check that the same system of equations is also satisfied by $(\E\alpha_+\otimes \E\alpha_+,\E\alpha_+\otimes \E\alpha_-,\E\alpha_-\otimes \E\alpha_+,\E\alpha_-\otimes \E\alpha_-)$, and since they share the same initial condition, the solution is unique, therefore we know $(\alpha_+,\alpha_-)$ is deterministic and satisfies \eqref{eq:tranEqsys}. The proof is complete.
\end{proof}

\subsection{Tightness}
In this section, we prove that $(\alpha_+^\eps,\alpha_-^\eps)$ is tight in $\C([0,T];H_M)$. So together with Proposition \ref{prop:1stMomentCon} and \ref{prop:2ndMomentCon}, we have finished the proof of the main theorem.

\begin{proposition}
$(\alpha_+^\eps,\alpha_-^\eps)$ is tight in $\C([0,T];H_M)$.
\label{prop:tightness}
\end{proposition}
\begin{proof}
By Proposition \ref{prop:tightcondition} we only have to show for any good function $f$ that the process $\langle f,\alpha_\pm^\eps\rangle \in \C([0,T];\R)$ is tight. Take $\alpha_+^\eps$ for example and note that
\begin{equation*}
\langle f,\alpha_+^\eps\rangle=\langle f(x_1x_1^*+x_2x_2^*),W_\eps\rangle.
 \end{equation*}
 Define $\lambda_0=f(x_1x_1^*+x_2x_2^*)$ so we have $\langle f,\alpha_+^\eps\rangle=\langle \lambda_0,W_\eps\rangle$. Recall that
\begin{equation*}
\begin{aligned}
\langle \lambda_0,W_\eps\rangle (t)&=\int_0^t ds \langle (\partial_t+\A)\lambda_0, W_\eps\rangle(s)+\G_{\lambda_0}^{1,\eps}[W_\eps](t)-\langle C_1,W_\eps\rangle(t)-\int_0^tds
\langle C_2+C_3+C_4,W_\eps\rangle(s)\\
&:=x_\eps(t)-y_\eps(t).
\end{aligned}
\end{equation*}
By our previous results, we have
\begin{equation*}
|y_\eps(t)|=|\langle C_1,W_\eps\rangle(t)+\int_0^tds
\langle C_2+C_3+C_4,W_\eps\rangle(s)|\leq C_{f,T} \sqrt{\eps}
 \end{equation*}
 uniformly in $t\in [0,T]$ and $\tilde{A}\in \V$, so $y_\eps(t)\Rightarrow 0$ in $\C([0,T];\R)$. Then we know that if $x_\eps(t)$ has a weakly convergent subsequence, so does $\langle \lambda_0,W_\eps\rangle(t)$, i.e., the relatively compactness of $x_\eps(t)$ implies relatively compactness of $\langle \lambda_0,W_\eps\rangle(t)$. Thus by the Prohorov theorem \cite{billingsley1968convergence} we only have to prove that $x_\eps(t)$ is tight in $\C([0,T];\R)$.

We apply the following sufficient conditions to show $x_\eps(t)$ is tight.
First, a Kolmogorov moment condition \cite{billingsley1968convergence} in the form
\begin{equation}
\E\{ |x_\eps(t)-x_\eps(t_1)|^\gamma
|x_\eps(t_1)-x_\eps(s)|^\gamma\} \leq C_{f,T} |t-s|^{1+\beta},
\label{eq:komocondition}
\end{equation}
with $0\leq s\leq t_1\leq t\leq T$ and $\gamma>0,\beta>0$. Second, we should have
\begin{equation*}
  \lim_{R\to \infty}\limsup_{\eps\to 0}\Pb\{\sup_{0\leq t\leq T}|x_\eps (t)|>R\}=0.
\end{equation*}

Since both $\langle \lambda_0,W_\eps\rangle(t)$ and $y_\eps(t)$ are uniformly bounded, the second condition is automatically satisfied, so we only have to prove \eqref{eq:komocondition} for
\begin{equation*}
x_\eps(t)=\int_0^t ds\langle (\partial_t+\A)\lambda_0,W_\eps\rangle(s)+\G_{\lambda_0}^{1,\eps}[W_\eps](t),
 \end{equation*}
 where $\G_{\lambda_0}^{1,\eps}[W_\eps](t)$ is a martingale. We have
 \begin{equation*}
 \begin{aligned}
 \E\{|x_\eps(t)-x_\eps(s)|^2|\F_s\}&\leq 2\E\{|\int_s^t d\tau \langle (\partial_t+\A)\lambda_0,W_\eps\rangle(\tau)|^2|\F_s\}+2\E\{|\G_{\lambda_0}^{1,\eps}[W_\eps](t)-
 \G_{\lambda_0}^{1,\eps}[W_\eps](s)|^2|\F_s\}\\
 &:=(I)+(II)
 \end{aligned}
 \end{equation*}
 and $(I)\leq C(t-s)^2$. To estimate $(II)$, we only have to calculate the increasing process associated with the sub-martingale $|\G_{\lambda_0}^{1,\eps}[W_\eps](t)|^2$. We check
\begin{equation*}
|\G_{\lambda_0}^{1,\eps}[W_\eps](t)|^2=\mbox{martingale part}+\frac{1}{\eps}\int_0^tds\left(
\Q|\langle \lambda_\eps,W_\eps\rangle|^2-\langle \lambda_\eps, W_\eps\rangle\langle W_\eps, \Q\lambda_\eps\rangle -\langle W_\eps, \lambda_\eps\rangle\langle \Q\lambda_\eps,W_\eps\rangle\right).
\end{equation*}
 Moreover, the integrand $\Q|\langle \lambda_\eps,W_\eps\rangle|^2-\langle \lambda_\eps, W_\eps\rangle\langle W_\eps, \Q\lambda_\eps\rangle -\langle W_\eps, \lambda_\eps\rangle\langle \Q\lambda_\eps,W_\eps\rangle$ is of order $\eps$ by Lemma \ref{lem:lambda1}, \ref{lem:lambda2} and the boundedness of $\Q$. Therefore, we have $(II)\leq C(t-s)$, which leads to
\begin{equation*}
\E\{|x_\eps(t)-x_\eps(s)|^2|\F_s\}\leq C(t-s).
\end{equation*}

In order to obtain \eqref{eq:komocondition},we note that
\begin{equation*}
\begin{aligned}
&\E\{|x_\eps(t)-x_\eps(t_1)|^\gamma|x_\eps(t_1)-x_\eps(s)|^\gamma\}
=\E\{\E\{|x_\eps(t)-x_\eps(t_1)|^\gamma|\F_{t_1}\}|x_\eps(t_1)-x_\eps(s)|^\gamma\}\\
\leq &
\E\{\left(\E\{|x_\eps(t)-x_\eps(t_1)|^2|\F_{t_1}\}\right)^{\frac{\gamma}{2}}|x_\eps(t_1)-x_\eps(s)|^\gamma\}
\leq C(t-t_1)^{\frac{\gamma}{2}}\E\{\E\{|x_\eps(t_1)-x_\eps(s)|^\gamma|\F_s\}\}\\
\leq &
C(t-t_1)^{\frac{\gamma}{2}}\E\{\left(\E\{|x_\eps(t_1)-x_\eps(s)|^2|\F_s\}\right)^{\frac{\gamma}{2}}\}
\leq C(t-t_1)^\frac{\gamma}{2}(t_1-s)^\frac{\gamma}{2}\leq C (t-s)^\gamma.
\end{aligned}
\end{equation*}
By choosing $\gamma\in (1,2)$, we finish the proof.
\end{proof}

\section{Slow time fluctuations}
\label{sec:slowerTime}
In this section, we discuss briefly the case when the temporal random fluctuations are slower than the spatial fluctuations of the potential. This leads to a kinetic regime with elastic collisions as the random fluctuations are now too slow to affect such a scattering; see \cite{bal2010kinetic} for a similar derivation in the setting of Schr\"odinger equations.

More precisely, we consider the Dirac equation \eqref{eq:diracEqnew} with $A_i(\frac{t}{\eps},\frac{x}{\eps})$ replaced by $A_i(\frac{t}{\eps^\alpha}, \frac{x}{\eps})$ for some $\alpha\in (0,1)$. The proof of convergence for $\alpha$ sufficiently large remains almost the same as before except that the infinitesimal generator of $(W_\eps(t),\tilde{A}(t/\eps^\alpha))$ becomes
\begin{equation}
\frac{d}{dh}\E^{\bar{\Pb}_\eps}_{W,\hat{V},t}\{
\langle \lambda(\hat{V}),W\rangle\}(t+h)|_{h=0}
=\langle \frac{1}{\eps^\alpha}\Q\lambda+\partial_t\lambda+\A_\eps\lambda,W\rangle.
 \end{equation}
So our test functions $\lambda_{1}, \lambda_{2}$ in \eqref{eq:lambda1} and \eqref{eq:lambda2} become
\begin{eqnarray*}
&&\lambda_1=-\int_{\R^d}\frac{e^{iz\cdot p}}{(2\pi)^d}\int_0^\infty e^{r\eps^{1-\alpha}\Q}e^{-rA_1}F_1e^{-rA_2}drdp,\\
&&\lambda_2=-\int_{\R^d}\frac{e^{iz\cdot p}}{(2\pi)^d}\int_0^\infty e^{r\eps^{1-\alpha}\Q}e^{-rA_1}(F_2+\Li_\eps\lambda_0(2\pi)^d\delta(p))e^{-rA_2}drdp,
\end{eqnarray*}
where $\Li_\eps\lambda_0$ is defined by replacing $\tilde{R}_{mn}(r,q)$ in $\Li\lambda_0$ by $\tilde{R}_{mn}(r\eps^{1-\alpha},q)$. Thus we see that $\lambda_{1,\eps}$ is of order $\eps^{\alpha-1}$ and $\lambda_{2,\eps}$ is of order $\eps^{2\alpha-2}$. To make those correctors $C_i$ small, we have to choose $\alpha\in (\frac{3}{4},1)$. Note that for $C_1,C_2,C_4$, $\alpha>\frac{1}{2}$ is enough, but for $C_3$, it has to be greater than $\frac{3}{4}$.

In the approximating martingale inequalities \eqref{eq:1stapproMartin} and \eqref{eq:2ndapproMartin}, we replace $\mathcal{A}\lambda_0$ by the $\eps-$dependent operator $\frac{1}{2}\sum_{k=1}^3(c\gamma^0\gamma^kD_k\lambda_0+\frac{1}{2}D_k\lambda_0 c\gamma^0\gamma^k)+\Li_\eps\lambda_0$. The rest of the proof is the same except that when we pass to the limit, we need to calculate $x_1^*\Li^*_\eps x_1+x_2^*\Li^*_\eps x_2$ and this leads to some $\eps-$dependent scattering operator of the form
\begin{equation*}
\begin{aligned}
\mathcal{T}_\eps(\alpha_+,\alpha_-)
=&\frac{e^2}{(2\pi)^d}\int_{\R^d}(\alpha_+(q)-\alpha_+(\xi))\sum_{k=0}^3\omega_k(\xi,q)\eps^{\alpha-1}\hat{R}_{kk}(\frac{c\lambda_+(q)-c\lambda_+(\xi)}{\eps^{1-\alpha}},q-\xi)dq\\
+&\frac{e^2}{(2\pi)^d}\int_{\R^d}(\alpha_-(q)-\alpha_+(\xi))\sum_{k=0}^3\tilde{\omega}_k(\xi,q)\eps^{\alpha-1}\hat{R}_{kk}(\frac{c\lambda_+(q)+c\lambda_+(\xi)}{\eps^{1-\alpha}},q-\xi)dq.
\end{aligned}
\end{equation*}
Following the same type of proof in \cite{bal2010kinetic} and letting $\eps \to 0$, we arrive at the elastic scattering operator
\begin{equation}
\mathcal{T}(\alpha_+)
=e^2\int_{\R^d}(\alpha_+(q)-\alpha_+(\xi))\sum_{k=0}^3\omega_k(\xi,q)\delta(c\lambda_+(q)-c\lambda_+(\xi))\tilde{R}_{kk}(0,q-\xi)dq.
\label{eq:elasticScatter}
\end{equation}
Therefore, in the slow temporal fluctuation case, when $\alpha\in (\frac{3}{4},1)$, we have the following transport equation system of the limit $(\alpha_+,\alpha_-)$:
\begin{equation}
\partial_t\alpha_\pm+\frac{c\xi\cdot \nabla_x\alpha_\pm}{\lambda_\pm(\xi)}=\mathcal{T}(\alpha_\pm),
\label{eq:tranEqSlow}
\end{equation}
where $\mathcal{T}$ is the elastic scattering operator defined in \eqref{eq:elasticScatter}.

We see that the coupling between $\alpha_+$ and $\alpha_-$ appeared in \eqref{eq:tranEqsys} is inactive, and we expect \eqref{eq:tranEqSlow} to hold in the limit of no time-dependent regularization, i.e., formally for $\alpha=0$.
\begin{remark}
The condition that $\alpha>\frac{3}{4}$ could be relaxed somewhat. As a matter of fact, if we construct the test function as
\begin{equation*}
\lambda_\eps=\lambda_0+\sum_{n=1}^N \eps^{\frac{n}{2}}\lambda_{n,\eps},
\end{equation*}
and follow the same procedure, we can show that $\lambda_{n,\eps}$ is of order $\eps^{n\alpha -n}$, and $\frac{n}{2}+n\alpha-n>0$ gives us $\alpha>\frac{1}{2}$. For the corrector
\begin{equation*}
C_3=\frac{1}{\sqrt{\eps}}ie\sum_{k=0}^3 \gamma^0\gamma^k\K_\eps^k \eps^{\frac{N}{2}}\lambda_{N,\eps}-\frac{1}{\sqrt{\eps}}ie\sum_{k=0}^3\tilde{\K}_\eps^k\eps^{\frac{N}{2}}\lambda_{N,\eps}\gamma^0\gamma^k,
\end{equation*}
it is of the order $\eps^{N\alpha-\frac{N+1}{2}}$, and $N\alpha-\frac{N+1}{2}>0$ leads to
\begin{equation*}
\alpha>\frac{1}{2}+\frac{1}{2N}.
\end{equation*}
Therefore, by expanding in higher order, we can relax the assumption to be $\alpha\in (\frac{1}{2},1)$. Note that $\alpha=\frac12$ corresponds to the regime considered in \cite{breteaux2011geometric}.

For slower time fluctuations of the media, i.e., when $\alpha< \frac{1}{2}$ or even $\alpha\to 0$, other techniques than the Markovian regularization considered here presumably need to be developed, and the use of the diagrammatic techniques as in \cite{erdHos2000linear,spohn1977derivation} is currently unavoidable.
\end{remark}

\section{Conclusion and further discussions}
\label{sec:conclusion}
In this paper, we derived the kinetic limit of the Dirac Equation with time-dependent random electromagnetic field. We have shown that the cross modes $c_{ij}^\eps, d_{ij}^\eps$ converged to zero weakly in space while the limiting propagating modes $\alpha_+=a_{11}+a_{22}$ and $\alpha_-=b_{11}+b_{22}$ satisfied a transport equation system. In addition, the temporal regularization brings some new features (inelastic scattering) to the scattering structure, which disappear when the random fluctuations are slower in time.

The method we use relies on the fact that the random field $\tilde{A}(t,p)$ is Markovian in $t$. By constructing appropriate test functions, we prove some approximating martingale inequalities, and together with the tightness result, we pass to the limit. We should mention that our approach is restricted to the $L^2$ case in the sense that the initial data should be appropriately generated such that $\|\Psi^\eps(0)\|_{L^2(\R^d)}^2$ is of order $\eps^{d/2}$.  In the setting of initial conditions $\|\Psi^\eps(0)\|_{L^2(\R^d)}^2$ of order $O(1)$, we do not expect convergence of the energy density to a deterministic limit; see \cite{bal2002radiative} in the setting of the scalar Schr\"odinger equation.

Some generalizations could be obtained by the same approach. We could for instance show that the complete set of modes $\{a_{ij}^\eps, b_{ij}^\eps\}$ converges weakly to the solution to a transport equation system of higher dimensions in \eqref{eq:bigtranEqsys}. The case when $\tilde{R}_{mn}\neq 0$ can also be handled by more strenuous computations. 

The same approach is expected to extend to linear hyperbolic systems with random coefficients such as those considered in \cite{ryzhik1996transport}. This is currently under study.

\section*{Acknowledgment} This paper was partially funded by AFOSR Grant NSSEFF- FA9550-10-1-0194 and NSF grant DMS-1108608.


\begin{thebibliography}{10}

\bibitem{bal2005kinetics}
{\sc G.~Bal}, {\em Kinetics of scalar wave fields in random media}, Wave
  Motion, 43 (2005), pp.~132--157.

\bibitem{B-CMP-09}
\leavevmode\vrule height 2pt depth -1.6pt width 23pt, {\em {Convergence to
  SPDEs in Stratonovich form}}, Comm. Math. Phys., 212(2) (2009), pp.~457--477.

\bibitem{bal1999radiative}
{\sc G.~Bal, A.~Fannjiang, G.~Papanicolaou, and L.~Ryzhik}, {\em Radiative
  transport in a periodic structure}, Journal of statistical physics, 95
  (1999), pp.~479--494.

\bibitem{bal2010kinetic}
{\sc G.~Bal, T.~Komorowski, and L.~Ryzhik}, {\em Kinetic limits for waves in a
  random medium}, Kinetic and Related Models, 3 (2010), pp.~529--644.

\bibitem{BKR-ARMA-11}
{\sc G.~Bal, T.~Komorowski, and L.~Ryzhik}, {\em {Asymptotics of the phase of
  the solutions of the random Schr{\"o}dinger equation}}, Archives for Rational
  Mechanics and Analysis, 200(2) (2011), pp.~613--664.

\bibitem{bal2002radiative}
{\sc G.~Bal, G.~Papanicolaou, and L.~Ryzhik}, {\em Radiative transport limit
  for the random schr{\"o}dinger equation}, Nonlinearity, 15 (2002), p.~513.

\bibitem{bal2002self}
\leavevmode\vrule height 2pt depth -1.6pt width 23pt, {\em Self-averaging in
  time reversal for the parabolic wave equation.}, Stochastics and Dynamics, 2
  (2002), pp.~507--531.

\bibitem{billingsley1968convergence}
{\sc P.~Billingsley}, {\em Convergence of probability measures}, New York,
  (1968).

\bibitem{blankenship1978stability}
{\sc G.~Blankenship and G.~Papanicolaou}, {\em Stability and control of
  stochastic systems with wide-band noise disturbances. i}, SIAM Journal on
  Applied Mathematics,  (1978), pp.~437--476.

\bibitem{breteaux2011geometric}
{\sc S.~Breteaux}, {\em A geometric derivation of the linear boltzmann equation
  for a particle interacting with a gaussian random field}, arXiv preprint
  arXiv:1107.0788,  (2011).

\bibitem{erdHos2000linear}
{\sc L.~Erd{\H{o}}s and H.~Yau}, {\em Linear boltzmann equation as the weak
  coupling limit of a random schr{\"o}dinger equation}, Communications on Pure
  and Applied Mathematics, 53 (2000), pp.~667--735.

\bibitem{fannjiang2005self}
{\sc A.~Fannjiang}, {\em Self-averaging scaling limits for random parabolic
  waves}, Archive for rational mechanics and analysis, 175 (2005),
  pp.~343--387.

\bibitem{gerard1997homogenization}
{\sc P.~G{\'e}rard, P.~Markowich, N.~Mauser, and F.~Poupaud}, {\em
  Homogenization limits and wigner transforms}, Communications on Pure and
  Applied Mathematics, 50 (1997), pp.~323--379.

\bibitem{gomez2011}
{\sc C.~Gomez}, {\em Radiative tranport limit for the random schr{\"o}dinger
  equation with long-range correlations}, to appear in Journal de
  Math{\'e}matiques Pures et Appliqu{\'e}es,  (2012).

\bibitem{guo1999transport}
{\sc M.~Guo and X.~Wang}, {\em Transport equations for a general class of
  evolution equations with random perturbations}, Journal of Mathematical
  Physics, 40 (1999), p.~4828.

\bibitem{KN-PA-10}
{\sc T.~Komorowski and E.~Nieznaj}, {\em On the asymptotic behavior of
  solutions of the heat equation with a random, long-range correlated
  potential}, Potential Analysis, 33(2) (2010), pp.~175--197.

\bibitem{lukkarinen2007kinetic}
{\sc J.~Lukkarinen and H.~Spohn}, {\em Kinetic limit for wave propagation in a
  random medium}, Archive for rational mechanics and analysis, 183 (2007),
  pp.~93--162.

\bibitem{papanicolaou1994functional}
{\sc G.~Papanicolaou and S.~Weinryb}, {\em A functional limit theorem for waves
  reflected by a random medium}, Applied mathematics \& optimization, 30
  (1994), pp.~307--334.

\bibitem{olivier2012}
{\sc O.~Pinaud}, {\em Classical limit for a system of random non-linear
  schr{\"o}dinger equations}, submitted,  (2012).

\bibitem{poupaud2003classical}
{\sc F.~Poupaud and A.~Vasseur}, {\em Classical and quantum transport in random
  media}, Journal de math{\'e}matiques pures et appliqu{\'e}es, 82 (2003),
  pp.~711--748.

\bibitem{powell2005transport}
{\sc J.~Powell and J.~Vanneste}, {\em Transport equations for waves in randomly
  perturbed hamiltonian systems, with application to rossby waves}, Wave
  motion, 42 (2005), pp.~289--308.

\bibitem{ryzhik1996transport}
{\sc L.~Ryzhik, G.~Papanicolaou, and J.~Keller}, {\em Transport equations for
  elastic and other waves in random media}, Wave motion, 24 (1996),
  pp.~327--370.

\bibitem{spohn1977derivation}
{\sc H.~Spohn}, {\em Derivation of the transport equation for electrons moving
  through random impurities}, Journal of Statistical Physics, 17 (1977),
  pp.~385--412.

\bibitem{GB2012schrodingerRandom}
{\sc N.~Zhang and G.~Bal}, {\em Convergence to spde of the schr{\"o}dinger
  equation with large, random potential}, in preparation,  (2012).

\bibitem{GB2012schrodingerDeterministic}
\leavevmode\vrule height 2pt depth -1.6pt width 23pt, {\em Homogenization of
  the schr{\"o}dinger equation with large, random potential}, submitted,
  (2012).

\end{thebibliography}
\end{document}